\documentclass[11pt]{amsart}
\usepackage{tikz-cd}
\usepackage{hhline}
\usepackage{comment}
\usepackage{tikz}
\usepackage[page]{appendix}
\usepackage[utf8]{inputenc}
\usepackage{filecontents}
\usepackage{graphicx} 
\usepackage{mathtools}
\usepackage{etoolbox}
\usepackage{amsmath}
\usepackage{amsthm}
\usepackage{amsfonts}
\usepackage{amssymb}
\usepackage{mathrsfs}
\usepackage{xypic}
\usepackage{color}
\usepackage{xcolor}
\usepackage{titlesec}
\usepackage{titletoc}
\usepackage{tabularx}
\usepackage{pgfplots}
\usepackage{subcaption}
\usepackage{listings}
\usepackage{kbordermatrix}

\definecolor{codegreen}{rgb}{0,0.6,0}
\definecolor{codegray}{rgb}{0.5,0.5,0.5}
\definecolor{codepurple}{rgb}{0.58,0,0.82}
\definecolor{backcolour}{rgb}{0.95,0.95,0.92}
\lstdefinestyle{mystyle}{
    backgroundcolor=\color{backcolour},   
    commentstyle=\color{codegreen},
    keywordstyle=\color{magenta},
    numberstyle=\tiny\color{codegray},
    stringstyle=\color{codepurple},
    basicstyle=\ttfamily\scriptsize,
    breakatwhitespace=false,         
    breaklines=true,                 
    captionpos=b,                    
    keepspaces=true,                 
    numbers=left,                    
    numbersep=5pt,                  
    showspaces=false,                
    showstringspaces=false,
    showtabs=false,                  
    tabsize=2
}
 
\usepackage{mathabx,graphicx}

\pgfplotsset{ticks=none}

\usepackage[colorlinks, linktocpage]{hyperref}
\definecolor{background}{RGB}{204,232,207}

\newtheorem{thm}{Theorem}

\usepackage[style=alphabetic,firstinits,backend=bibtex, uniquename=init,
url=true,hyperref]{biblatex}
\DeclareFieldFormat{labelalpha}{\thefield{entrykey}}
\DeclareFieldFormat{extraalpha}{}
\DeclareFieldFormat{postnote}{#1}
\DeclareFieldFormat{multipostnote}{#1}
\usepackage[colorlinks]{hyperref}
\definecolor{cclr}{rgb}{25,25,112}
\hypersetup{
citecolor=[rgb]{0.15, 0.15, 0.68},
linkcolor=gray,   
urlcolor=magenta}

\nocite{*}
\bibliography{lyndon-g2}

\newcommand{\disc}{\operatorname{disc}}
\newcommand{\Aut}{\operatorname{Aut}}
\newcommand{\red}{\operatorname{red}}
\newcommand{\ur}{\operatorname{ur}}
\newcommand{\pr}{\operatorname{pr}}

\newcommand{\univ}{\operatorname{univ}}
\newcommand{\crys}{\operatorname{crys}}

\newcommand{\std}{\operatorname{std}}

\newcommand{\gr}{\operatorname{gr}}
\newcommand{\LD}{\operatorname{LD}}

\newcommand{\res}{\operatorname{res}}

\newcommand{\spec}{\operatorname{Spec}}

\newcommand{\Span}{\operatorname{span}}

\newcommand{\invlim}[1]{\underset{#1}{\underleftarrow{\operatorname{lim}}}}

\titleformat{\section}[runin]{\normalfont\bfseries}{\thesection.}{3pt}{}
\titleformat{\subsection}[runin]{\normalfont\bfseries}{\thesubsection.}{3pt}{}
\titleformat{\subsubsection}[runin]{\normalfont\bfseries}{\thesubsubsection.}{3pt}{}
\titleformat{\paragraph}[runin]{\normalfont\bfseries}{\theparagraph}{3pt}{}
\renewcommand{\thesection}{\arabic{section}}
\titleformat{\section}{\normalfont\large\bfseries}{\thesection.~~}{1em}{}
\dottedcontents{section}[2.5em]{}{2.9em}{1pc}

\begin{document}

\newcommand{\DemRel}{{x_0^q(x_0,x_1)(x_2,x_3)\dots(x_n,x_{n+1})}}
\newcommand{\chicyc}{\chi_{\operatorname{cyc}}}
\newcommand{\Brnr}{\operatorname{Br}_{\operatorname{nr}}}
\newcommand{\Bronr}{\operatorname{Br}_{1,\operatorname{nr}}}
\newcommand{\Br}{\operatorname{Br}}
\newcommand{\Tr}{\operatorname{tr}}
\newcommand{\Mod}{\operatorname{Mod}}
\newcommand{\Qp}{\mathbb{Q}_p}
\newcommand{\TODO}{{\color{red} TODO}}
\newcommand{\Gm}{\mathbb{G}_m}
\newcommand{\Ga}{\mathbb{G}_a}
\newcommand{\Scin}{\operatorname{Scin}}
\newcommand{\Fil}{\operatorname{Fil}}
\newcommand{\Ind}{\operatorname{Ind}}
\newcommand{\Sym}{\operatorname{Sym}}
\newcommand{\sym}{\operatorname{sym}}
\newcommand{\semis}[1]{#1\operatorname{-ss}}
\newcommand{\Alt}{\operatorname{Alt}}
\newcommand{\DGK}{\operatorname{DG}_K}
\newcommand{\Img}{\operatorname{Im}}
\newcommand{\Ker}{\operatorname{Ker}}
\newcommand{\codim}{\operatorname{codim}}
\newcommand{\rank}{\operatorname{rank}}
\newcommand{\grp}{\operatorname{grp}}
\newcommand{\cont}{\operatorname{cont}}
\newcommand{\ord}{\operatorname{ord}}
\newcommand{\Hom}{\operatorname{Hom}}
\newcommand{\Ext}{\operatorname{Ext}}
\newcommand{\Ad}{\operatorname{Ad}}
\newcommand{\Id}{\operatorname{id}}
\newcommand{\Lie}{\operatorname{Lie}}
\newcommand{\Lift}{\operatorname{Lift}}
\newcommand{\ad}{\operatorname{ad}}
\newcommand{\Det}{\operatorname{Det}}
\newcommand{\LHS}{\operatorname{LHS}}
\newcommand{\RHS}{\operatorname{RHS}}
\newcommand{\bFp}{\bar{\mathbb{F}}_p}
\newcommand{\bZp}{\bar{\mathbb{Z}}_p}
\newcommand{\bQp}{\bar{\mathbb{Q}}_p}
\newcommand{\fP}{{\mathfrak{P}}}
\newcommand{\fL}{{\mathfrak{L}}}
\newcommand{\fU}{{\mathfrak{U}}}
\newcommand{\bF}{{\mathbb{F}}}
\newcommand{\bZ}{{\mathbb{Z}}}
\newcommand{\bA}{{\mathbb{A}}}
\newcommand{\bB}{{\mathbb{B}}}
\newcommand{\bQ}{{\mathbb{Q}}}
\newcommand{\bR}{{\mathbb{R}}}
\newcommand{\bC}{{\mathbb{C}}}
\newcommand{\scrC}{{\mathscr{C}}}
\newcommand{\Fq}{\bar{\mathbb{F}}_q}
\newcommand{\GSp}{\operatorname{GSp}}
\newcommand{\Sp}{\operatorname{Sp}}
\newcommand{\GL}{\operatorname{GL}}
\newcommand{\SL}{\operatorname{SL}}
\newcommand{\SO}{\operatorname{SO}}
\newcommand{\GO}{\operatorname{GO}}
\newcommand{\cO}{\mathcal{O}}
\newcommand{\lsup}[2]{{^{#1}\!#2}}
\newcommand{\wh}[1]{\widehat{#1}}
\newcommand{\wt}[1]{\widetilde{#1}}

\newcommand{\cX}{\mathcal{X}}
\newcommand{\Lieg}{\mathfrak{g}}
\newcommand{\Lieh}{\mathfrak{h}}
\newcommand{\Mat}{\operatorname{Mat}}
\newcommand{\Gal}{\operatorname{Gal}}
\newcommand{\Art}{\operatorname{Art}}
\newcommand{\BHT}{\mathbb{B}_{\operatorname{HT}}}
\newcommand{\DXT}{{\hatI}^{X(T)}}
\newcommand{\HTGC}{\prod_{\sigma:L'\hookrightarrow \bC}X_*(G_{\bC})}
\newcommand{\HTGCa}{\prod_{\tilde\tau\in S}X_*(G_{\bC})}
\newcommand{\HT}{{\operatorname{HT}}}
\newcommand{\dR}{{\operatorname{dR}}}
\newcommand{\hatbI}{\widehat{\bar I}_K}
\newcommand{\hatI}{\widehat{I}_K}
\newcommand{\symt}{\operatorname{sym}^2}
\newcommand{\altt}{\operatorname{alt}^2}
\newcommand{\socle}{\operatorname{soc}}
\newcommand\nlcup{%
  \mathrel{\ooalign{\hss$\cup$\hss\cr%
  \kern0.7ex\raise0.6ex\hbox{\scalebox{0.4}{$\diamond$}}}}}
\newcommand{\matt}[9]{
\left(
\begin{matrix}
#1 & #2 & #3 \\
#4 & #5 & #6 \\
#7 & #8 & #9
\end{matrix}
\right)
}

\author{Lin, Zhongyipan}
\title{Lyndon-Demu\v{s}kin method and crystalline lifts of $G_2$-valued Galois representations}

\begin{abstract}
We develop obstruction theory for lifting characteristic $p$
local Galois representations valued in reductive groups
of type $B_l$, $C_l$, $D_l$ or $G_2$.
An application of the Emerton-Gee stack then reduces
the existence of crystalline lifts
to a purely combinatorial problem when $p$ is not too small.

As a toy example, we show for all local fields $K/\Qp$, with $p >3$,
all representations $\bar\rho:G_K \to G_2(\bFp)$
admit a crystalline lift
$\rho: G_K\to G_2(\bZp)$, where $G_2$ is the exceptional Chevalley
group of type $G_2$.



\end{abstract}

\maketitle

\tableofcontents

\section{Introduction}

Let $K/\Qp$ be a p-adic field.
Let $G$ be a connected reductive group over $\bZp$.
Let $\bar\rho:G_K\to G(\bFp)$ be a Galois representation.

We will study whether there exist crystalline lifts 
of $\bar\rho$ to $G(\bZp)$.
This question has been raised in multiple papers,
for example,
(i) irreducible geometric Galois representations
\cite{NCS18}, (ii) the Serre weight conjecture \cite{GHS18},
(iii) ramification theory
 \cite{CL11}.

The pursuit of constructing characteristic $0$ lifts of Galois representations
(at least in higher dimensions)
is, however, resistant to elementary techniques.
\cite{B03} is able to lift mod $\varpi$ representations to a mod $\varpi^2$ one,
for $G=\GL_N$.
\cite{Mu13} constructed crystalline lifts of mod $\varpi$ represntations valued in
$G=\GL_3$, and recently
\cite{EG23} worked the $GL_N$-case for all $N$.
Our earlier work \cite{L22} answers this question for semisimple representations
valued in general reductive groups $G$.

The method of \cite{EG23} is purely local, and is based on an analysis of Galois cohomology.
The image group $\bar\rho(G_K)$ is either an irreducible subgroup of
$G(\bFp)$ 
or factors through a proper maximal parabolic $P$ of $G$.
In the former case,
our previous work \cite{L22} shows
$\bar\rho$ always admits a crystalline lift.
In this paper, we focus on the latter case.
Let $P = L\rtimes U_P$ be the Levi decomposition.
Let $\bar r:G_K\xrightarrow{\bar\rho} P(\bFp)\to L(\bFp)$
be the Levi factor of $\bar\rho$.
Then $\bar\rho$ defines a 1-cocycle
$[\bar c] \in H^1(G_K, U_P(\bFp))$.
What we will actually do is to
construct a lift $r:G_K\to L(\bZp)$ of $\bar r$ and a lift $[c]\in H^1(G_K, U_P(\bZp))$
of $[\bar c]$.

In the $GL_N$-case, all maximal proper parabolics have abelian unipotent radical, so it
suffices to consider abelian cohomology.
When $G$ is not $GL_N$,
parabolic subgroups with
abelian unipotent radical are rare.
For example, when $G$ is the
exceptional group $G_2$,
all parabolics have non-abelian unipotent radical.

Fortunately, for groups of type $A$, $B$, $C$, $D$ or $G_2$,
the relevant non-abelian Galois cohomology can be replaced by
abelian Galois cohomology equipped with a cup product structure
and the strategy considered in \cite{EG23} can be adapted to work.
In this paper, we focus on the $G_2$-case, and
prove the following theorem:

\begin{thm} [Theorem \ref{thm:existence-crys-lift-G2}]
\label{thm:C}
Assume $p>3$.
Every mod $\varpi$ Galois representation
valued in the exceptional group $G_2$
$$
\bar\rho:G_K\to G_2(\bFp)
$$
admits a crystalline lift
$\rho^{\circ}:G_K\to G_2(\bZp)$.

Moreover, if $\bar\rho$ factors through 
a maximal parabolic $P=L\ltimes U$
and the Levi factor $\bar r_{\bar\rho}:G_K\to L(\bFp)$ of $\bar\rho$
admits a Hodge-Tate regular and crystalline lift $r_1:G_K\to L(\bZp)$ such that
the adjoint representation $G_K \xrightarrow{r_1} L(\bZp) \to \GL(\Lie (U(\bZp)))$ has Hodge-Tate weights
slightly less than $\underline 0$ (Definition \ref{def:slightly-less}),
then $\rho^{\circ}$ can be chosen such that it factors through
the maximal parabolic $P$ and its Levi factor $r_{\rho^{\circ}}$
lies on the same irreducible component of the spectrum of the crystalline
lifting ring that $r_1$ does.
\end{thm}

\subsection{Overview of the method and comparison with \cite{L23}}
\label{par:classical-groups}

To establish the existence of crystalline lifts, we proceed in four steps:
\begin{itemize}
\item[(Step 1)] construct explicit cochain complexes {\it equipped with a natural cup product structure} that compute abelian Galois cohomology,

\item[(Step 2)] show that the cup product considered in (Step 1) is non-trivial in certain special cases,

\item[(Step 3)] compute the dimension of certain substacks of the reduced Emerton-Gee stack,

\item[(Step 4)] invoke the machinary of \cite{EG23} to produce crystalline lifts.
\end{itemize}

After the first draft of this paper was written,
we have a more conceptual understanding of some constructions made in this paper;
see the introduction section of \cite{L23}.
For example, 
section \ref{sec:LD} and section \ref{sec:manipulate} of this paper are conceptualized 
under the notion of {\it Heisenberg equations}.
In loc. cit., we also establish the existence of de Rham lifts for many classical groups
and in particular the existence of crystalline lifts for unramified unitary groups.

However, from the technical perspective, loc. cit. parallels this paper, instead of
upgrades this paper.
In loc. cit., we use {\it Herr complexes} as the explicit cochain complex computing Galois cohomology.
Herr complexes are infinite dimensional cochain complexes and are often not amenable to computation by hand.
We can truncate Herr complexes to a finite dimensional cochain complex but the truncation
can't be made explicit in general.
The upside of Herr complexes is better functoriality and in the case of classical groups,
we can usually reduce the problems to the $\GL_n$-case, which is well-understood.

In this paper, we use Lyndon's cochain complexes instead. 
Everything in this paper are totally explicit and are computable by hand or by a computer algebra system.
The downside of this approach is that the complexity of computation grows exponentially,
and quickly becomes out of hand for large-ranked classical groups.

We don't know how to deal with Herr complexes for exceptional groups because of their implicit nature,
and the approach in this paper is still the only one we are aware of.
In this paper, we establish the existence of crystalline lifts for the exceptional group $G_2$,
which illustrates the usefulness of Lyndon's cochain complexes.
Because of its explicit nature, our approach can potentially be extended to deal with more general exceptional groups,
after upgrading the cup product structure to more complicated higher Massey product structures.

\subsection{Obstruction theory for crystalline lifting}
~
\vspace{3mm}

In this paper, we consider the case
where $U_P$ admits a quotient $U$ such that
\begin{itemize}
\item The adjoint group $U^{\ad}:=U/Z(U)$
is abelian;
\item The center $Z(U)$ is isomorphic to $\Ga$; and
\item There is a bijection of obstructions
``$H^2(G_K, U_P(\bFp))$''
$\cong$ ``$H^2(G_K, U(\bFp))$''.
\end{itemize}
We call $U$ a Heisenberg quotient of $U_P$.
When $G$ is of type $B_l$, $C_l$, $D_l$ or $G_2$,
it is always possible to choose a parabolic
$P$ whose unipotent radical admits
a Heisenberg quotient (see subsection \ref{par:classical-groups}).


Let $\spec R$ be an irreducible component of a crystalline lifting ring $\spec R^{\crys,\underline{\lambda}}_{\bar r}$ (Definition \ref{def:crys-lifting})
of $\bar r$.
Let $r^{\univ}:G_K\to L(R)$
be the universal family.
The Levi factor group acts on $U$
via conjugation
$\phi:L\to\Aut(U)$.
Write
$\phi^{\ad}:L\to\GL(U^{\ad})$
and $\phi^z:L\to\GL(Z(U))$
for the graded pieces of $\phi$.

The theorem we prove is:

\begin{thm}[\ref{thm:heisenberg-lift}]
\label{thm:A}
Let $[\bar c]\in H^1(G_K, U(\bF))$
be a characteristic $p$ cocycle,
where $U$ is a Heisenberg quotient of $U_P$.

Assume
\begin{itemize}
\item[ {[1]} ]
$H^2(G_K, \phi^{\ad}(r^{\univ}))$ is sufficiently generically regular (Definition \ref{def:sgr})
and set-theoretically supported on the special fiber of $\spec R$;
\item[ {[2]} ] $p\ne 2$;
\item[ {[3]} ]
There exists a 
finite Galois extension $K'/K$ of prime-to-$p$ degree
such that
$\phi(\bar r)|_{G_{K'}}$
is Lyndon-Demu\v{s}kin (Definition \ref{def:ld}); and
\item[ {[4]} ]
There exists a $\bZp$-point
of $\spec R$ which is mildly regular (Definition \ref{def:mildly-regular})
when restricted to $G_{K'}$.
\end{itemize}
Then there exists a $\bZp$-point of
$\spec R$ which gives rise to
a Galois representation
$r^{\circ}:G_K\to L(\bZp)$ such that
if we endow $U(\bZp)$ with the
    $G_K$-action
    $G_K\xrightarrow{r^{\circ}} L(\bZp)\xrightarrow{\phi} \Aut(U(\bZp))$,
the cocycle $[\bar c]$ has a characteristic $0$
lift $[c]\in H^1(G_K, U(\bZp))$.
\end{thm}

\noindent \text{\bf Remark}
[3] is automatically satisfied if $p$ is sufficiently large;
and [4] is automatically satisfied if $p$ is sufficiently large and the labeled Hodge-Tate weights
$\phi^{\ad}(\underline \lambda)$ are
slightly less then $0$ (Definition \ref{def:slightly-less}).

\subsubsection{Example: $G=\GL_3$}
~

Let $\bar \rho:G_K\to \GL_3(\bFp)$ 
be a completely reducible Galois representation.
There are two ways of encoding the
data of $\bar\rho$ as a $1$-cocycle in
Galois cohomology.

(I) Use the fact $\bar\rho$ factors through
a maximal parabolic 
$$
P=
\begin{bmatrix}
* &* &* \\
* &* &*\\
0 &0 &*
\end{bmatrix}
=
\begin{bmatrix}
* &* &0 \\
* &* &0\\
0 &0 &*
\end{bmatrix}
\ltimes
\begin{bmatrix}
1 &0 &* \\
0 &1 &*\\
0 &0 &1
\end{bmatrix}
=L\ltimes A
$$
where $A\cong \Ga^{\oplus 2}$
is a rank-$2$ abelian group.
Let $\bar r:G_K\xrightarrow{\bar\rho}P(\bFp)\to L(\bFp)$
be the Levi factor of $\bar\rho$.
The information of $\bar\rho$
is encoded in a $1$-cocycle
$[\bar c]\in H^1(G_K, \phi(\bar r))=:H^1(G_K, A(\bFp))$.
We first construct a lift $r^{\circ}:G_K\to\GL_2(\bZp)$
of $\bar r$.
Then we construct a lift $[c]\in H^1(G_K, A(\bZp))$ of $[\bar c]$.

(II) Use the fact $\bar\rho$ factors through a Borel
(minimal parabolic)
$$
B=
\begin{bmatrix}
* &* &* \\
0 &* &*\\
0 &0 &*
\end{bmatrix}
=
\begin{bmatrix}
* &0 &0 \\
0 &* &0\\
0 &0 &*
\end{bmatrix}
\ltimes
\begin{bmatrix}
1 &* &* \\
0 &1 &*\\
0 &0 &1
\end{bmatrix}
=T\ltimes H
$$
where the Levi group 
$T$ is a maximal torus,
and the unipotent radical 
$H$ is the Heisenberg group.
Let $\bar r:G_K\to T(\bFp)$
be the Levi factor of $\bar\rho$.
To reconstruct $\bar\rho$
from $\bar r$, 
we only need the information
of a $1$-cocycle
$[\bar c]\in H^1(G_K, H(\bFp))$.
We first construct a lift of $\bar r$,
and then construct
a lift of $\bar c$.
Now $H^1(G_K, H(\bFp))$
is non-abelian Galois cohomology.

We make use of the graded
structure of $\Lie H$ when we
construct a lift of $[\bar c]$.
We have a short exact sequence
$$
1\to
\begin{bmatrix}
1 &0 &* \\
0 &1 &0\\
0 &0 &1
\end{bmatrix}
\to H \to
\begin{bmatrix}
1 &* & \\
0 &1 &*\\
0 &0 &1
\end{bmatrix}
\to 1.
$$
We will first construct a lift
modulo $Z(H)$, and then extend
the lift modulo $Z(H)$
to a cocycle on the whole
unipotent radical $H$.

Theorem \ref{thm:A} applies
in this situation,
so we have a new proof 
for the group $\GL_3$.

\subsubsection{}
We have a short exact sequence of
groups
$0\to Z(U)\to U\to U^{\ad}\to 0$.
Since $Z(U)$ is a central, normal subgroup, we have a long exact sequence of pointed sets
$$
H^1(G_K, Z(U))\to H^1(G_K, U)\to H^1(G_K, U^{\ad})\xrightarrow{\delta} H^2(G_K, Z(U)).
$$
Note that $\delta$ is a quadratic form, and 
there is an associated
bilinear form 
$$
\cup: H^1(G_K, U^{\ad}) \times H^1(G_K, U^{\ad})
\to H^2(G_K, Z(U))
$$
defined by $x\cup y = (\delta(x+y)-\delta(x)-\delta(y))/2$.

The technical heart of this paper is an analysis of $\cup$
on the cochain/cocycle level.
So we need a finite cochain complex computing
Galois cohomology which interacts nicely with 
the bilinear form $\cup$.
Thanks to the theory of Demu\v{s}kin groups,
there is an explicitly defined cochain complex
(the so-called Lyndon-Demu\v{s}kin complex)
which computes $H^\bullet(G_{K'},U^{\ad})$
and $H^\bullet(G_{K'},Z(U))$
after a finite Galois extension $K'/K$.
When $[K':K]$ is prime to $p$,
we can fully understand cup products
on the cochain/cocycle level
via Lyndon-Demu\v{s}kin complexes
endowed with $G_K/G_{K'}$-action.

We have the following nice obstruction theory:

\begin{thm} [\ref{cor:obstruction-theory}]
\label{thm:B}
Let $p\ne 2$ be a prime integer.
Let $L$ be a reductive group over $\cO_E$ and fix
an algebraic group homomorphism $L\to\Aut(U)$.
Let $r :G_K \to L(\cO_E)$ be a Galois representation.

If there exists a
finite Galois extension $K'/K$ of prime-to-$p$ degree
such that
$r|_{G_{K'}}$ is Lyndon-Demu\v{s}kin and mildly regular, then
there is a short exact sequence of pointed sets
$$
H^1(G_K, U(\bZp))
\to H^1(G_K, U(\bFp))
\xrightarrow{\delta} H^2(G_{K}, U^{\ad}(\bZp))
$$
where $\delta$ has a factorization
$H^1(G_K, U(\bFp))\xrightarrow{p}
H^1(G_{K}, U^{\ad}(\bFp)) \to
H^2(G_{K}, U^{\ad}(\bZp))$.
\end{thm}

\vspace{3mm}



\subsection{Organization}

In section \ref{sec:LD}, we review 
the results of Lyndon and Demu\v{s}kin
and establish some notations.

Section \ref{sec:FC}
and Section \ref{sec:manipulate}
form the technical heart of this paper.

Section \ref{sec:MACHINE} 
and Section \ref{sec:codim}
are mild generalizations of results 
from \cite{EG23}.
The proof is almost unchanged and 
we often just sketch the ideas of the proof
and invite the readers to look at the proofs
of \cite{EG23}.

We prove the main theorem in Section \ref{sec:main}.

\subsection{Acknowledgement}

I would like to thank David Savitt,
for suggesting to me the project of constructing
    crystalline lifts of Galois representations
    valued in general reductive groups,
    and for his excellent advisoring.
    I would like to thank Matthew Emerton
    for teaching me his work \cite{EG23}.
I also want to thank the referees for very careful reading
and
thank Joel Specter and Xiyuan Wang
for helpful discussions.

\vspace{3mm}

\section{Lyndon-Demu\v{s}kin theory}

\label{sec:LD}

Assume $p\ne 2$.

Let $K/\Qp$ be a finite extension
containing the $p$-th root of unity.
The maximal pro-$p$ quotient of
the absolute Galois group $G_K$ has
a very nice description.
The following well-known theorem can be found,
for example, in \cite[Section II.5.6]{Se02}.

\subsubsection{Theorem}
Let $G_K(p)$ be the maximal pro-$p$ quotient
of $G_K$. Then $G_K(p)$
is the pro-$p$ completion of the
following one-relator group
$$
\langle x_0,\cdots,x_{n+1}|\DemRel\rangle
$$
where $n=[K:\Qp]$, and $q=p^s$ is the largest power of $p$
such that $K$ contains the $q$-th roots of unity.
Here $(x,y)=xyx^{-1}y^{-1}$.

\subsubsection{Definition}
\label{def:ld}
A continuous profinite $G_K$-module $A$ is
said to be \textit{Lyndon-Demu\v{s}kin}
if the image of $G_K\to\Aut(A)$
    is a pro-$p$ group.

\subsection{Comparing cohomology of Demu\v{s}kin groups
and Galois cohomology}~

\vspace{3mm}

Let $\Gamma^{\disc}$
be the discrete group
with one relator
$$
\langle
x_0, \dots, x_n, x_{n+1}|\DemRel
\rangle.
$$

Let $K/\Qp$ be a $p$-adic field
containing the group of $p$-th root of
 unity.
Let $A$ be a Lyndon-Demu\v{s}kin $G_K$-module.
Write $H^{\bullet}(\Gamma^{\disc}, A)$
for the usual group cohomology,
and write
$H^{\bullet}(G_K, A)$
for the continuous profinite cohomology.

Note that there is a functorial map
$$
\text{($\dagger$)}\hspace{2mm}H^{\bullet}(G_K, A) \to H^{\bullet}(\Gamma^{\disc}, A)
$$
induced from the forgetful functor
$\Mod_{\cont}(G_K(p))\to \Mod(\Gamma^{\disc})$.

\subsubsection{Lemma}
Let $\bF_p$ be the $G_K$-module with
trivial $G_K$-action.
Then
($\dagger$) induces isomorphisms:

(1) $H^1(G_K, \bF_p)=H^1(\Gamma^{\disc}, \bF_p)$;

(2) $H^2(G_K, \bF_p)=H^2(\Gamma^{\disc}, \bF_p)$;

\begin{proof}
(1)
We have
$$H^1(G_K, \bF_p) = \Hom_{\cont}(G_K, \bF_p)=\Hom_{\cont}(G_K(p), \bF_p);$$
$$H^1(\Gamma^{\disc}, \bF_p)=
\Hom(\Gamma^{\disc}, \bF_p).$$
Note that $\Hom_{\cont}(G_K(p), \bF_p)= \Hom(\Gamma^{\disc}, \bF_p)$
because $G_K(p)$ is the pro-$p$ completion of $\Gamma^{\disc}$.

(2)
We have a commutative diagram
$$
\xymatrix{
    H^1(G_K, \bF_p)
    \times H^1(G_K, \bF_p)
    \ar@<-6ex>[d]
    \ar@<+9ex>[d]\ar[r]^{\hspace{9mm}\cup}
    &
    H^2(G_K, \bF_p)\ar[d]\\
    H^1(\Gamma^{\disc}, \bF_p)
    \times H^1(\Gamma^{\disc}, \bF_p)
    \ar[r]^{\hspace{9mm}\cup}
    &
    H^2(\Gamma^{\disc}, \bF_p)
}
$$
Note that the first row
is a non-degenerate pairing,
and $H^2(G_K, \bF_p)\cong \bF_p$
by local Tate duality.
By Lyndon's theorem or
Corollary \ref{cor:h2-disc},
we have $H^2(\Gamma^{\disc}, \bF_p)\cong \bF_p$.
So it remains to show the cup product
of the second row is non-trivial.
Let $[c_1], [c_2]\in H^1(\Gamma^{\disc}, \bF_p)$.
$[c_1]\cup[c_2] = 0$ if and only
if there exists a group homomorphism
    $$
    \Gamma^{\disc}\to
\begin{bmatrix}
1 & c_1 & * \\
 & 1 & c_2 \\
 & & 1
\end{bmatrix}
    $$
    for some $*$.
        Indeed, if $c_1\cup c_2= d z$ for some $z\in C^1(\Gamma^{\disc}, \bF_p)$,
        then $\Gamma^{\disc}\to
\begin{bmatrix}
1 & c_1 & z \\
 & 1 & c_2 \\
 & & 1
\end{bmatrix}$ is a group homomorphism by unravelling the definition of cup products;
here $C^1(\Gamma^{\disc}, \bF_p)$ is the usual cochain group defining group cohomology.
        Define $c_i:\Gamma^{\disc}\to\bF_p$ by sending $x_i$ to $1$ and other generators to $0$, $i=0,1$.
        Then it is clear $[c_1]\cup[c_2]\ne 0$.
\end{proof}

\subsubsection{Corollary}
Let $A$ be a finite $\bF_p$-vector
space endowed with 
Lyndon-Demu\v{s}kin $G_K$-action.
Then there is a canonical isomorphism
$H^{\bullet}(G_K, A) = H^{\bullet}(\Gamma^{\disc}, A)$.

\begin{proof}
Let $G_K(p)$ be the maximal pro-$p$
quotient of $G_K$.
Then $A$ is a $G_K(p)$-module.
Since $G_K(p)$ is a pro-$p$ group,
$A$ must contain the trivial representation $\bF_p$.
In particular, there is a short exact sequence
$$
0\to \bF_p\to A \to A'\to 0
$$
which induces the long exact sequence
$$
\xymatrix{
H^0(G_K, A')\ar[r]\ar[d]&
H^1(G_K, \bF_p) \ar[r]\ar[d]&
H^1(G_K, A) \ar[r]\ar[d]&
H^1(G_K, A') \ar[r]\ar[d]&
H^2(G_K, \bF_p) \ar[d]\\
H^0(\Gamma^{\disc}, A')\ar[r]&
H^1(\Gamma^{\disc}, \bF_p) \ar[r]&
H^1(\Gamma^{\disc}, A) \ar[r]&
H^1(\Gamma^{\disc}, A') \ar[r]&
H^2(\Gamma^{\disc}, \bF_p)\\
}
$$
We apply induction on the length
of $A$.
By the Five Lemma, we have 
$H^1(G_K, A)=H^1(\Gamma^{\disc}, A)$.

We also have the long exact sequence
$$
\xymatrix{
H^1(G_K, A')\ar[r]\ar[d]&
H^2(G_K, \bF_p) \ar[r]\ar[d]&
H^2(G_K, A) \ar[r]\ar[d]&
H^2(G_K, A') \ar[r]\ar[d]&
H^3(G_K, \bF_p) \ar[d]\\
H^1(\Gamma^{\disc}, A')\ar[r]&
H^2(\Gamma^{\disc}, \bF_p) \ar[r]&
H^2(\Gamma^{\disc}, A) \ar[r]&
H^2(\Gamma^{\disc}, A') \ar[r]&
H^3(\Gamma^{\disc}, \bF_p)\\
}
$$
By Lyndon's theorem,
$H^3(\Gamma^{\disc}, \bF_p)=0$.
By local Tate duality,
$H^3(G_K, \bF_p)=0$.
Again by the Five Lemma, we have
$H^2(G_K, A)=H^2(\Gamma^{\disc}, A)$.
Finally, both cohomology groups are supported on degrees $[0,2]$.
\end{proof}

By induction on the order of $A$,
($\dagger$) is an isomorphism for
any finite $p$-power torsion group $A$.

\subsubsection{Corollary}
\label{cor:cont-disc}
Let $A$ be a finite $\bZ_p$-module
endowed with Lyndon-Demu\v{s}kin $G_K$-action.
Then there is a canonical isomorphism
$
H^{\bullet}(G_K, A)=H^{\bullet}(\Gamma^{\disc}, A)
$.

\begin{proof}
We have a short exact sequence
for each $k>0$,
$$0\to \invlim{i}^1H^{k-1}(G_K, A/p^i A)
\to H^k(G_K, A)\to \invlim{i} H^k(G_K, A/p^iA)\to 0,$$
see, for example \cite[Tag 0BKN]{stacks-project}; here $\invlim{i}^1$ is the derived inverse limit.
The first term
is $0$ due to the finiteness
of the cohomology of
torsion $G_K$-modules.
So $H^k(G_K, A)=\invlim{i}H^k(G_K, A/p^iA)$,
and the corollary is reduced to the $p$-power torsion case.

We can do the same thing for
the discrete cohomology.
Since any finite $\bZ_p$-module
is $p$-adically complete,
the Lyndon-Demu\v{s}kin complex
(see the last subsection of Section \ref{sec:LD})
computing
$H^{\bullet}(\Gamma^{\disc}, A)$
is the inverse limit
of the Lyndon-Dem\"iskin complex
mod $p^i$.
So $H^k(\Gamma^{\disc}, A)=\invlim{i}H^k(\Gamma^{\disc}, A/p^i)$.
\end{proof}

The lemma above tells us that,
for our purposes, the cohomology groups of $G_K(p)$
can be computed via the discrete model.
So we can make use of the fine machineries of
combinatorial group theory.

\subsection{Discrete group cohomology of Demu\v{s}kin groups} ~

\vspace{3mm}

The main reference of this subsection is \cite{Ly50}.

\vspace{3mm}

\paragraph{Derivations}

A derivation of a group $G$ is a left $G$-module $M$,
together with a map
$D: G\to M$ such that $D(uv) = Du+uDv$.

Say $F$ is a free group with generators $x_1$,\dots $x_m$.
Denote by $dFJ$ the module of universal derivations.
Then $dFJ$ is the free $\mathbb{Z}[F]$-module
with basis $\{dx_i|i=1,\dots,m\}$.

Let $u\in F$.
We can write $du\in dFJ$ as a linear combination of the basis elements:
$
du = \sum \frac{\partial u}{\partial x_i}
dx_i
$
where $\frac{\partial u}{\partial x_i}\in \mathbb{Z}[F]$.
The computation rules for $\frac{\partial u}{\partial x_i}$ can be found
in the first line of page 654 of \cite{Ly50}.

\vspace{3mm}

\paragraph{Theorem} (Lyndon, \cite[Theorem 11.1]{Ly50})
Let $
G = \langle x_1,\dots,x_m | R \rangle
$ be a one-relator group
where $R=Q^q$ for no $q>1$.
Let $K$ be any left $G$-module. Then
$$
H^2(G,K) \cong 
K/ (\frac{\partial R}{\partial x_1},\dots
\frac{\partial R}{\partial x_m})K
$$
and
$
H^n(G,K) = 0
$
for all $n>2$.

\vspace{3mm}

\paragraph{Corollary}
\label{cor:h2-disc}
We have
$H^2(\Gamma^{\disc}, \bF_p)=\bF_p$.

\begin{proof}
We have the following computation:
\begin{eqnarray}
\frac{\partial R}{\partial x_0}
&=&
1+x_0+\cdots+x_0^{q-2} + x_0^{q-1}x_1^{-1}
\nonumber\\
\frac{\partial R}{\partial x_1}
&=&
x_0^{q-1}x_1^{-1}(x_0-1)
\nonumber\\
\frac{\partial R}{\partial x_2}
&=&
x_0^q(x_0,x_1)x_2^{-1}(x_3-1)
\nonumber\\
\frac{\partial R}{\partial x_3}
&=&
x_0^q(x_0,x_1)x_2^{-1}x_3^{-1}(x_2-1)
\nonumber\\
\vdots& & \vdots\nonumber\\
\frac{\partial R}{\partial x_{2k}}
&=&
x_0^q(x_0,x_1)\cdots(x_{2k-2},x_{2k-1})
x_{2k}^{-1}(x_{2k+1}-1)
\nonumber\\
\frac{\partial R}{\partial x_{2k+1}}
&=&
x_0^q(x_0,x_1)\cdots(x_{2k-2},x_{2k-1})
x_{2k}^{-1}x_{2k+1}^{-1}(x_{2k}-1)
\nonumber\\
\vdots&& \vdots\nonumber
\end{eqnarray}

Since
$H^2(\Gamma^{\disc}, \bF_p)=
\frac{\bF_p}{(\partial R/\partial x_0, \cdots, \partial R/x_{n+1})}$,
it suffices to show
$$
\frac{\partial R}{\partial x_0}\bF_p=\dots=
\frac{\partial R}{\partial x_{n+1}}\bF_p=0.
$$
Since $\bF_p$ is a trivial $G_K$-module,
it is clear
$\frac{\partial R}{\partial x_1}\bF_p=\dots
\frac{\partial R}{\partial x_{n+1}}\bF_p
=0$.
We also have
$
\frac{\partial R}{\partial x_0}=
1+1+\cdots+1=q=0
$ mod $p$.
\end{proof}

\subsubsection{Proposition}
\label{prop:coh}
Let $A$ be a $G_K$-module
whose underlying abelian group 
is a finitely generated $\mathbb{Z}_p$-module such that
the image of $G_K$ in $\Aut(A)$ is a pro-$p$ group.
Then
$$
H^2(G_K,A)\cong 
A/ (\frac{\partial R}{\partial x_0},\dots,
\frac{\partial R}{\partial x_{n+1}})A
$$
where
$
R = 
\DemRel
$.

\begin{proof}
Combine Corollary \ref{cor:cont-disc} and Lyndon's theorem.
\end{proof}

\vspace{3mm}

\subsection{Lyndon-Demu\v{s}kin Complex}~
\label{ss:LDC}

\subsubsection{Abelian coefficient case}
Let $A$ be a $G_K$-module
whose underlying abelian group 
is a finitely generated $\mathbb{Z}_p$-module such that
the the image of $G_K$ in $\Aut(A)$ is a pro-$p$ group.

Then there is an explicit co-chain complex 
computing the Galois cohomology $H^{\bullet}(G_K,A)$.

Define 
$C^{\bullet}_{\LD}(A)
=
[C^0_{\LD}(A)\xrightarrow{d^1} C^1_{\LD}(A)
\xrightarrow{d^2} C^2_{\LD}(A)]$
as
the following cochain complex supported
on degrees [0,2]
$$
A \xrightarrow{
\begin{bmatrix}
1-x_0\\
\dots\\
1-x_{n+1} \\
\end{bmatrix}
}
A^{\oplus (n+2)} \xrightarrow{
\begin{bmatrix}
\partial R/\partial x_0\\
\dots\\
\partial R/\partial x_{n+1}\\
\end{bmatrix}^{T}
}A.
$$

Then by \cite[Theorem 11.1]{Ly50}
$$
H^{\bullet}(C^{\bullet}_{LD}(A))
=H^{\bullet}(G_K, A).
$$

The idea of Lyndon Demu\v{s}kin
complex is simple.
A $1$-cochain $c\in C^1_{\LD}(A)$
is simply a set-theoretical function
$$
c: \{x_0, \dots, x_{n+1}\}
\to A.
$$
We can extend $c$ to be a function
on the free group
$$
c:\langle x_0, \dots, x_{n+1}\rangle
\to A
$$
by setting
$c(gh):=c(g)+g\cdot c(h)$
for any $g,h$ in the free group
    with $(n+2)$ generators.
Let 
$$
R=\DemRel
$$
be the single relation defining the
Demu\v{s}kin group.
The differential operator
$d^2:C^1_{\LD}(A)\to C^2_{\LD}(A)$
is nothing but the evaluation of the
extended map $c$ at the relation $R$,
that is,
$d^2(c) = c(R)$.
So a $1$-cochain $c$ is a $1$-cocycle
if and only if its evaluation at
    $R$ is $0$.

\subsubsection{Nilpotent coefficients}
\label{sss:nilcoef}

Let $E/\Qp$ be a finite extension
with ring of integers $\cO_E$,
residue field $\bF$, and uniformizer $\varpi$.

Let $U$ be a unipotent (smooth connected) linear algebraic group
over $\spec \cO_E$, admitting
an upper central series
$$
{1} = U_0 \subset U_1 \cdots \subset U_k
= U.
$$
Assume there exists an embedding $\iota: U\hookrightarrow \GL_N\subset \Mat_{N\times N}$
such that $(\iota(x)-1)^{k+1}=0$ for all $x\in U$.
Write $\log=\log_{\le k}$ for the truncated logarithmic function
$1+x\mapsto x - x^2/2 + \cdots + (-1)^{k+1}x^k/k$.

Assume $p>k$.
There is an isomorphism of
schemes
$
U \cong \Lie U
$
sending $g\mapsto \log g$, defined through the following commutative diagram
$$
\xymatrix{
    U\ar[d]^{\log}\ar[r] & \GL_N \ar[d]^{\log}\\
    \Lie U \ar[r] & \Mat_{N\times N}
}
$$

We assume $k=2$ from now on
because it suffices for our applications.


Fix a Galois action
$G_K\to \Aut(U)(\cO_E)$
such that the image group is a pro-$p$
subgroup of $\Aut(U)(\cO_E)$.

Let $A$ be an $\cO_E$-algebra.
Recall that a \textit{non-abelian crossed homomorphism} valued in $U(A)$
is a map $c:G_K\to U(A)$ such that
$$
c(gh) = c(g)(g\cdot c(h))
$$
for all $g,h\in G_K$.
Set $\mathfrak{c}:=\log (c): G_K \to \Lie U(A)$.
By the Baker-Campbell-Hausdorff formula,
$$
\text{($\dagger$)}\hspace{5mm}\mathfrak{c}(gh) = \mathfrak{c}(g) + 
g\cdot \mathfrak{c}(h) + \frac{1}{2}[
\mathfrak{c}(g), g\cdot\mathfrak{c}(h)].
$$
Our definition
of the Lyndon-Demu\v{s}kin cochain complex
is motivated by ($\dagger$).

\vspace{3mm}

\paragraph{Definition}
\label{def:LD-diff}
Let $A$ be an $\cO_E$-algebra.
The Lyndon-Demu\v{s}kin complex with unipotent coefficients is defined to be the following cochain complex
$C_{\LD}^{\bullet}(U(A))$
supported in degrees [0,2]:
$$
\Lie U(A) \xrightarrow{d^1} (\Lie U(A))^{\oplus n+2}
\xrightarrow{d^2} \Lie U(A)
$$
where $d^1$ is defined by
$$
d^1(v) 
= (-v + x_i\cdot v + \frac{1}{2}[-v, x_i\cdot v])_{i=0,\ldots,n+1}.
$$
We need some preparations
before we define $d^2$.
An element
$c = (\alpha_0, \cdots, \alpha_{n+1})\in C^1_{\LD}(U(A))$
can be regarded as a function on the free group with $(n+2)$ generators
$$
c: \langle x_0, \cdots, x_{n+1} \rangle \to \Lie U(A)
$$ by setting $c(x_i)=\alpha_i$ for each $i$
and extending it to the whole free group
by
$$c(gh):=c(g)+g\cdot c(h)+\frac{1}{2}[c(g), g\cdot c(h)]$$
We define $d^2$ as
$$
d^2 (c):= c(R) = c(\DemRel).
$$

\paragraph{Remark}
\label{rem:ld-underlying-group}
(1)
When $U$ is an abelian group,
we recover the definition
in the previous section
for the cohomology of the abelian $U(A)$;

(2)
The main reason we define $C^{\bullet}_{\LD}(U(A))$ this way
is because we want to compare it with
$C^{\bullet}_{\LD}(\Lie U(A))$.

Note that
$C^{\bullet}_{\LD}(\Lie U(A))$
and 
$C^{\bullet}_{\LD}(U(A))$
have the same underlying group,
but their differential $d^{\bullet}$
is different.

(3) Note that $d^2(c) = 0$
if and only if $c$ defines a crossed homomorphism $\mathfrak{c}:G_K\to \Lie U(A)$
    in the sense of ($\dagger$).
    See the proof of Proposition \ref{prop:HLDvsH}.

(4) The differential maps are generally non-linear.

\vspace{3mm}

\paragraph{Definition}
We define $Z_{\LD}^{i}:=(d^{i+1})^{-1}(0)$,
and $B_{\LD}^i := d^i(C_{\LD}^{i-1})$ for $i=0,1,2$.

\vspace{3mm}
\paragraph{Proposition}
\label{prop:HLDvsH}
We have $$H^0(G_K, U(A))\cong Z^0_{\LD}(U(A))$$
and a surjection of pointed sets
$$
Z^1_{\LD}(U(A)) \to H^1(G_K, U(A)).
$$

\begin{proof}
$H^0(G_K, U(A))$ is by definition the $G_K$-fixed point subset
of $U(A)$,
while $Z^0_{\LD}(U(A))$ is the subset of $U(A)$ whose elements
are fixed by the $x_0$,\ldots, $x_{n+1}$:
if $u\in U(A)$ is fixed by $x_i$, then
    $u^{-1}(x_i\cdot u)=1$ and taking truncated log of both sides
    we get $d^1(\log u)=0$.

$H^1(G_K, U(A))$ is by definition the set of equivalence classes of crossed homomorphisms,
and $Z^1_{\LD}(U(A))$ is the set of crossed homomorphisms.
\end{proof}


\vspace{3mm}

$\Lie U$ has a lower central series
filtration.
Let $Z(U)$ be the center of $U$.
Write $U^{\ad}$ for $U/Z(U)$.
Since $U$ is unipotent of class $2$,
$\Lie U$ is isomporphic to its graded
Lie algebra
$\Lie U\cong \gr^{\bullet}(\Lie U)$.
We will fix a grading 
$\Lie U \cong Z(U) \oplus U^{\ad}$
of the Lie algebra $\Lie U$
once for all.
In particular, we fixed 
a projection 
$\pr:\Lie U \to Z(U)$.

\subsubsection{Cup products}
\label{def:LD-cup}
Let $c\in C^1_{\LD}(U^{\ad}(A))$.
Let $\widetilde{c} \in C^1_{\LD}(U(A))$
be the (unique) lift of $c$ such that
$
\pr(\widetilde{c}(x_0))=\dots
\pr(\widetilde{c}(x_{n+1}))=
0$.
Define $$Q(c):= \pr(d^2(\widetilde{c}))
=\pr(\widetilde{c}(R))
\in C^2_{\LD}(Z(U)(A)).$$

\noindent
{\bf Lemma}
$Q(-)$ is a quadratic form, that is, 
$(x, y)\mapsto Q(x+y)-Q(x)-Q(y)$ is a bilinear form.

\begin{proof}
In Definition \ref{def:LD-diff},
we defined it so that
$\wt c(gh):=\wt c(g)+g\cdot \wt c(h)+\frac{1}{2}[\wt c(g), g\cdot \wt c(h)]$.
So after fully expanding the expression, $\wt c(R) = \sum_i \alpha_i c(x_i) + \sum_{i<j}[\beta_i c(x_i), \gamma_i c(x_j)]$,
where $\alpha_i, \beta_i, \gamma_j\in \langle x_0, \ldots, x_{n+1}\rangle$.
Thus $Q(c)=\pr(\sum_i \alpha_i c(x_i) + \sum_{i<j}[\beta_i c(x_i), \gamma_i c(x_j)]) = \sum_{i<j}\pr([\beta_i c(x_i), \gamma_i c(x_j)])$, which is clearly a quadratic form.
\end{proof}

We define
\begin{eqnarray}
C_{\LD}^1(U^{\ad}(A))\times
C_{\LD}^1(U^{\ad}(A)) &\xrightarrow{\cup}&
C_{\LD}^2(Z(U)(A))\nonumber\\
x\cup y &:= & \frac{1}{2}(Q(x+y)-Q(x)-Q(y))\nonumber
\end{eqnarray}
which is a symmetric bilinear form.

\noindent
\text{\bf Remark}
Alternatively, we can choose an arbitrary lift
$\widetilde{c}$ of $c$.
Now $\pr(d^2(\widetilde{c}))$ is an inhomogeneous
polynomial of degree two.
We recover $Q$ by taking the homogeneous part
of degree two.

\vspace{3mm}

\paragraph{Lemma}
\label{lem:non-abelian-h1}
Under the identification
 $
 C^1_{\LD}(U(A))=
C^1_{\LD}(U^{\ad}(A)) \oplus C^1_{\LD}(Z(U)(A))
$, we have
$$
Z^1_{\LD}(U(A)) = 
\{(x, y)\in
C^1_{\LD}(U^{\ad}(A)) \oplus C^1_{\LD}(Z(U)(A))|
d^2x = 0,~x\cup x + d^2 y = 0
\}.
$$

\begin{proof}
It is obvious from the definition
of $d^2$
and $Q$.
The projection of $d^2(x,y)$ to $C^2_{\LD}(U^{\ad}(A))$ is $d^2x$; and
the projection of $d^2(x,y)$ to $C^2_{\LD}(Z(U)(A))$ is $x\cup x + d^2y$.
\end{proof}

Write $H^i_{\LD}(U^{\ad}(A))$
for $$Z^i_{\LD}(U^{\ad}(A))/B^i_{\LD}(U^{\ad}(A))$$
    and write
$H^i_{\LD}(Z(U)(A))$
for $$Z^i_{\LD}(Z(U)(A))/B^i_{\LD}(Z(U)(A)).$$

    \vspace{3mm}

\paragraph{Lemma}
\label{lem:h1-cup}
The pairing $\cup$ on the cochain level 
induces a symmetric pairing on the cohomology level
\begin{eqnarray}
H_{\LD}^1(U^{\ad}(A))\times
H_{\LD}^1(U^{\ad}(A)) &\xrightarrow{\cup}&
H_{\LD}^2(Z(U)(A)).\nonumber
\end{eqnarray}

\begin{proof}
It suffices to show for all
$x\in Z^1_{\LD}(U^{\ad})(A)$
and $y\in B^1_{\LD}(U^{\ad})(A)$,
$Q(x+y)-Q(x)\in B^2_{\LD}(Z(U)(A))$.

Let $\widetilde{x}\in C^1_{\LD}(U(A))$
be the unique extension of $x$
such that $\pr \widetilde{x}= 0$.
The cochain $\widetilde{x}$
represents a group homomorphism
$\rho_{\widetilde{x}}: \langle x_0,\cdots,x_{n+1}\rangle \to U(A)\rtimes \langle x_0,\cdots,x_{n+1}|R\rangle$
such that $\rho_{\widetilde{x}}(R)=1$ mod $Z(U)(A)$.
More explicitly,
we define $\rho_{\widetilde{x}}(x_i)=(\exp(\wt x(x_i)), x_i)$ where $\exp$ is the truncated exponential map (the inverse to the truncated log map).
Since $y$ is a coboundary,
there exists $n\in U(A)$
such that
$n \rho_{\widetilde{x}} n^{-1}$
is represented by a cocycle $(x+y,f)$ extending $x+y$
(we are exploiting the abelian coefficients here).
We have
$n \rho_{\widetilde{x}}(R) n^{-1} \rho_{\widetilde{x}}(R)^{-1}=1\in U(A)\rtimes \langle x_0,\cdots,x_{n+1}|R\rangle$
since $\rho_{\widetilde{x}}(R)$ lies in the center of $U(A)$.
Since $Q(x+y)-d^2(f)=n \rho_{\widetilde{x}}(R) n^{-1}$
and $Q(x)=\rho_{\widetilde{x}}(R)$,
we have $Q(x+y)-Q(x)=d^2f\in B^2_{\LD}(Z(U)(A))$.
\end{proof}

\vspace{3mm}
Recall $Z^1_{\LD}(U(A))$ and
$Z^1_{\LD}(\Lie U(A))$ are both subsets of $C^1_{\LD}(U(A))$.
\paragraph{Lemma}
\label{lem:incl-into-lie-u}
If $Z(U)(\bF)\cong \bF$, then
$$
Z^1_{\LD}(U(\bF)) \subset
Z^1_{\LD}(\Lie U(\bF))
$$
that is, the non-abelian cocycles
with $U(\bF)$-coefficients
are automatically
abelian cocycles with $(\Lie U(\bF))$-coefficients.

\begin{proof}
We have remarked in \ref{rem:ld-underlying-group}(2) that
$C^1_{\LD}(U(\bF))$ and $C^1_{\LD}(\Lie U(\bF))$
have the same underlying space.
By Lemma 
\ref{lem:non-abelian-h1},
an element of $Z^1_{\LD}(U(\bF))$
is a pair $(x,y)$ such that $d^2x=0$
and $x\cup x+d^2 y = 0$.
By our assumption,
$C^2_{\LD}(Z(U)(\bF)) = H^2(G_K, Z(U)(\bF))$
(Corollary \ref{cor:h2-disc})
and thus $B^2_{\LD}(Z(U)(\bF))=0$ and $d^2=0$.
So $d^2y=0$ automatically,
and $(x,y)$ defines an element
of $Z^1_{\LD}(\Lie U(\bF))$.
\end{proof}

\vspace{3mm}

\section{An analysis of cup products}

\label{sec:FC}

Let $E$ be a $p$-adic field with ring of integers $\cO_E$, residue field $\bF$ and uniformizer $\varpi$.

Let $U$ be a smooth connected unipotent group of class $2$ over $\spec \cO_E$,
with center $Z(U)\cong \Ga$.
Write $U^{\ad}$ for $U/Z(U)$.
Assume $U^{\ad}\cong \Ga^{\oplus s}$ is a vector group.

\vspace{3mm}

\subsubsection{Definition}
\label{def:mildly-regular}
Let $K'$ be a $p$-adic field.
A Lyndon-Demu\v{s}kin action $G_{K'}\to \Aut(U)(\cO_E)$
is said to be \textit{mildly regular}
if the following are satisfied:
\begin{itemize}
\item[(MR1)] $H^0(G_{K'}, U^{\ad}(E)) = 0$;
\item[(MR2)] The bilinear pairing
$$
\cup_{\bF}:
C^1_{\LD}(U^{\ad}(\bF))
\times
C^1_{\LD}(U^{\ad}(\bF))
\to C^2_{\LD}(Z(U)(\bF))
$$
is non-degenerate.
\end{itemize}

\paragraph{\bf Remark}
In practice $U$ is the unipotent radical of a parabolic subgroup of a reductive group and (MR2) is equivalent to ``$p$ being not too small''.
We worked out the $G_2$-case in Appendix \ref{sec:delta}, and showed that if $p>5$, (MR2) always holds.
The same proof but with more complicated notation should work for general reductive groups.

In general, (MR2) can be checked by computer algebra systems because
it is a finite field vector space question
for a finite number of small $p$'s.
We include an algorithm (written in SageMath) in Appendix \ref{sec:sage}.

The following proposition is a summary of Appendix \ref{sec:delta}:

\vspace{3mm}

\text{\bf Proposition}
\label{prop:summary-appendix}
If $U$ is the unipotent radical of the short root parabolic of $G_2$ or the quotient of the unipotent radical
of the long root parabolic of $G_2$ by its center,
then (MR2) is true when $p\ge 5$.
\vspace{3mm}

\subsubsection{Definition}
\label{def:slightly-less}
Given a tuple of labeled Hodge-Tate weights (see \cite[Subsection 1.12]{EG23} for the definition)
$\underline \lambda$,
we say $\underline \lambda$
is slightly less than $0$
if for each $\sigma:K'\hookrightarrow \bQp$,
$\lambda_{\sigma}$ consists of 
non-positive integers,
and for at least one $\sigma$,
$\lambda_\sigma$ consists of negative integers.
(The cyclotomic character has Hodge-Tate weight $-1$.)

\vspace{3mm}
\subsubsection{Proposition}
\label{prop:G2-MR}
Assume $p\ge 5$.
If $U$ is the unipotent radical of the short root parabolic of $G_2$ or the quotient of the unipotent radical
of the long root parabolic of $G_2$ by its center,
then $G_{K'}\to \Aut(U)(\cO_E)$ is mildly regular
if $U^{\ad}(E)$ is Hodge-Tate of
    labeled Hodge-Tate weights slightly
    less then $0$.

\begin{proof}
If $H^0(G_{K'}, U^{\ad}(E)) \ne 0$,
then for all embeddings $\sigma:K\hookrightarrow \bQp$, $0\in \lambda_\sigma$.

The proposition now follows from 
the Proposition in Remark
\ref{prop:summary-appendix}
and Appendix \ref{sec:delta}.
\end{proof}

\subsection{Cup products mod $\varpi$}

\subsubsection{Lemma}
The image of
$Z^1_{\LD}(U^{\ad}(\cO_E)) \to
C^1_{\LD}(U^{\ad}(\bF))$
has codimension at most $\dim_{E} U^{\ad}(E)$.

\begin{proof}
Say $\dim_{\bF} C^1_{\LD}(U^{\ad}(\bF))=
\rank_{\cO_E}C^1_{\LD}(U^{\ad}(\cO_E))=N$.

Since $Z^1_{\LD}(U^{\ad}(\cO_E))$
is the kernel of
$C^1_{\LD}(U^{\ad}(\cO_E))
\to C^2_{\LD}(U^{\ad}(\cO_E))$,
and $\rank_{\cO_E}C^2_{\LD}(U^{\ad}(\cO_E))=\dim_E U^{\ad}(E)$,
we have
$$
\rank_{\cO_E} Z^1_{\LD}(U^{\ad}(\cO_E))\ge N-\dim_E U^{\ad}(E).
$$
Since $C^2_{\LD}(U^{\ad}(\cO_E))$
is torsion-free,
$Z^1_{\LD}(U^{\ad}(\cO_E))$
is saturated in $C^1_{\LD}(U^{\ad}(\cO_E))$,
and is thus a direct summand.
In particular,
the image of $Z^1_{\LD}(U^{\ad}(\cO_E))$
in $C^1_{\LD}(U^{\ad}(\bF))$
has dimension $\ge N-\dim_E U^{\ad}(E)$.
\end{proof}

\subsubsection{Lemma}
\label{lem:char-p-ker-estimate}
If 
$$
\cup_{\bF}:
C^1_{\LD}(U^{\ad}(\bF))
\times
C^1_{\LD}(U^{\ad}(\bF))
\to C^2_{\LD}(Z(U)(\bF))
$$
is non-degenerate, then
the kernel of 
$$
\cup_{\bF}:
Z^1_{\LD}(U^{\ad}(\cO_E))/\varpi
\times
Z^1_{\LD}(U^{\ad}(\cO_E))/\varpi
\to C^2_{\LD}(Z(U)(\bF))
$$
has dimension at most $\dim_E U^{\ad}(E)$.

\vspace{3mm}
\text{\bf Remark}
Note that $Z^1_{\LD}(U^{\ad}(\bF))\ne Z^1_{\LD}(U^{\ad}(\cO_E))/\varpi$ in general.

The kernel of a bilinear pairing is also called the annihilator.

\begin{proof}
For ease of notation, write
$C$ for $C^1_{\LD}(U^{\ad}(\bF))$,
and write 
$Z$ for the image of
$Z^1_{\LD}(U^{\ad}(\cO_E))$
in $C$.
Note that $Z\cong Z^1_{\LD}(U^{\ad}(\cO_E))/\varpi$ by the proof of the above lemma.

Let $K\subset Z$ be the kernel of $\cup_{\bF}$.
Since the cup product on $C$
is non-degenerate, there exists
a subspace $F\subset C$
of dimension equal to that of $K$,
such that the restriction of the cup product to $(F + K)$
is also non-degenerate.
Since $F \cap Z = 0$,
$\dim C \ge \dim (F + Z) = \dim Z + \dim F = \dim Z + \dim K$.
The lemma now follows from the previous lemma.
\end{proof}

We also record the following lemma whose proof is similar.

\subsubsection{Lemma}
\label{lem:char-p-ker-estimate-variant}
(1)
The image of
$Z^1_{\LD}(U^{\ad}(\bF)) \to
C^1_{\LD}(U^{\ad}(\bF))$
has codimension at most $\dim_{E} U^{\ad}(E)$.

(2)
If 
$$
\cup_{\bF}:
C^1_{\LD}(U^{\ad}(\bF))
\times
C^1_{\LD}(U^{\ad}(\bF))
\to C^2_{\LD}(Z(U)(\bF))
$$
is non-degenerate, then
the kernel of 
$$
\cup_{\bF}:
Z^1_{\LD}(U^{\ad}(\bF))
\times
Z^1_{\LD}(U^{\ad}(\bF))
\to C^2_{\LD}(Z(U)(\bF))
$$
has dimension at most $\dim_E U^{\ad}(E)$.

\subsection{General cup products in group cohomology}~

In this subsection, we give a reinterpretation of Definition \ref{def:LD-cup},
which is convenient for theoretic applications.

Let $V$ be a unipotent algebraic group of class $2$ over $\cO_E$.
Let $\Gamma$ be an abstract group, together with a homomorphism $\theta: \Gamma \to \Aut(V)(\cO_E)$.
By the Lie correspondence,
$\Aut(\Lie V)\cong \Aut(V)$,
and thus $\theta$ induces
a $\cO_E$-linear $\Gamma$-action on $\Lie V$
which respects Lie brackets.

We fix a grading $\Lie V = V_1 \oplus V_2$ such that $[V_1, V_1] \subset V_2$, and $[V, V_2]=0$.
We will write $V$ for $V(\cO_E)$ for simplicity.

Let $f: \Gamma \to V$ be a crossed homomorphism.
By definition, for any $g_1,g_2\in \Gamma$, $f(g_1g_2)=f(g_1)g_1f(g_2)$.
Write $c=c_1+c_2$ for $\log(f)$, where $c_1$ values in $V_1$ and $c_2$ values in $V_2$.
By the Baker-Campbell-Hausdorff formula, we have
\begin{align*}
\text{($*$)}\hspace{3mm}c(gh)&=c(g)+g c(h) + [c(g), g c(h)]/2\\
&=(c_1(g)+ g c_1(h)) + (c_2(g)+g c_2(h)) + [c_1(g), g c_1(h)]/2
\end{align*}

\subsubsection{Lemma}
Let $a, b \in H^1(\Gamma, V_1)$ be two crossed homomorphisms.
The $2$-cochain $B(a, b): (g,h)\mapsto [a(g), g b(h)]$
is a $2$-cocycle.

\begin{proof}
By definition, we have
\begin{align*}
d^2(B(a, b))(g_1, g_2, g_3) &= 
g_1[a(g_2), g_2 b(g_3)] -
[d^1a(g_1, g_2), g_1g_2b(g_3)]
\\
&\hspace{5mm}+[a(g_1), g_1d^1b(g_2, g_3)]
+[a(g_1), g_1 b(g_2)]\\
&=g_1[a(g_2), g_2 b(g_3)] -
[a(g_1)+ g_1 a(g_2), g_1g_2b(g_3)]
\\
&\hspace{5mm}+[a(g_1), g_1b(g_2) + g_1g_2 b(g_3)]
+[a(g_1), g_1 b(g_2)]\\
&=0 \qedhere
\end{align*}
\end{proof}

For crossed homomorphisms
$a\in H^1(\Gamma, V_1)$,
define $Q(a):= B(a,a)$.
By comparing $(*)$
and paragraph \ref{def:LD-cup},
it is not hard to see
the $Q(-)$ defined in this
subsection coincides
with that of \ref{def:LD-cup}
for $1$-cocycles
when $\Gamma$ is the discrete Demu\v{s}kin group.

Since $a\cup b:= (Q(a+b, a+b)-Q(a)-Q(b))/2=(B(a,b) + B(b,a))/2$, we have $a\cup b\in H^2(\Gamma, V_2)$.
Again the cup product defined in this subsection coincides with
the \ref{def:LD-cup} when the settings overlap.

\subsubsection{Lemma}
\label{lem:equivariance}
Let $\Gamma' \subset \Gamma$ be
a normal subgroup of finite index.
Write $\Delta$ for $\Gamma/\Gamma'$.

The cup product $\cup:H^1(\Gamma', V_1)\times H^1(\Gamma', V_1)\to H^2(\Gamma', V_2)$ is $\Delta$-equivariant.

\begin{proof}
Let $a, b\in H^1(\Gamma^1, V_1)$, and let 
$\sigma\in \Gamma$.
We have by definition $\sigma\cdot a(g)=\sigma a(\sigma^{-1} g \sigma)$, and
$\sigma\cdot B(a,b)(g, h)=\sigma B(a,b)(\sigma^{-1} g\sigma, \sigma^{-1} h \sigma)$ (see \cite[Section I.5.8]{Se02}).
We immediately have $\sigma \cdot B(a,b) = B(\sigma\cdot a, \sigma \cdot b)$.
\end{proof}

\subsubsection{Example: the completely split case}
\label{example:trivial}
In this paragraph we analyze the special case where
the $G_{K'}$ action on $U^{\ad}(\bF)\cong \Lie U^{\ad}(\bF)$ is trivial
and $H^2(G_{K'}, Z(U)(\bF))=Z(U)(\bF)=\bF$.
It will be used in the proof of Theorem \ref{thm:nontrivial}.

Since the center of $\Lie U$
is one-dimensional,
the Lie bracket
$$
\Lie U^{\ad}(\bF) \times
\Lie U^{\ad}(\bF) \xrightarrow{[-,-]} Z(U)(\bF)
$$
is a non-degenerate, alternating pairing.
Choose a basis $\{e_1, \cdots, e_k, e_1', \cdots, e_k'\}$
of $\Lie U^{\ad}(\bF)$
such that $[e_i', e_j']=[e_i,e_j]=0$ and $[e_i, e_j']=-[e_i', e_j]=\delta_{i,j}$.
Since by assumption the $G_{K'}$-action on $U^{\ad}(\bF)$
is trivial, the cup product
$$
\cup:
H^1(G_{K'}, U^{\ad}(\bF))
\times
H^1(G_{K'}, U^{\ad}(\bF))
\to
H^2(G_{K'}, Z(U)(\bF))
$$
is isomorphic to
the (exterior) direct sum of 
cup products
$$
\cup_i:
H^1(G_{K'}, \bF e_i \oplus \bF e_i')
\times
H^1(G_{K'}, \bF e_i \oplus \bF e_i')
\to
H^1(G_{K'}, \bF)
$$
Write 
$\wedge$
for the usual cup product
$H^1(G_{K'}, \bF) \times H^1(G_{K'}, \bF) \to H^2(G_{K'}, \bF)$ which appears in local Tate duality.
By definition, 
for $a, b\in H^1(G_{K'}, \bF)$
we have
\begin{align*}
Q(a e_i + b e_i') &= B(a e_i + b e_i', a e_i + b e_i')\\
& =((g,h) \mapsto [a(g) e_i + b(g)e_i', a(h) e_i + b(h)e_i'])\\
&= ((g, h)\mapsto (a(g)b(h)-b(g)a(h))\\
&= a \wedge b - b \wedge a \\
&= 2 a \wedge b
\end{align*}
and thus
for $a_1, b_1, a_2, b_2\in H^1(G_{K'}, \bF)$
$$
B(a_1 e_i+ b_1 e_i', a_2 e_i + b_2 e_i')
= 2 (a_1\wedge b_2 + a_2 \wedge b_1)
$$
Since $\wedge$ is a non-degenerate pairing, $B$ is also a non-degenerate pairing.

\subsection{Nontriviality of cup products}~

\subsubsection{Theorem}
\label{thm:nontrivial}
Let $K'/K$ be a finite Galois extension of $p$-adic fields of prime-to-$p$ degree.
Let $r :G_K \to \Aut(U)(\cO_E)$ be a continuous group homomorphism.

If $r|_{G_{K'}}$ is Lyndon-Demu\v{s}kin and mildly regular, then
one of the following are true:
\begin{itemize}
\item[(i)]
$H^2(G_K, Z(U)(\bF)) = 0$, or
\item[(ii)] the symmetric bilinear pairing
$$
H^1(G_K, U^{\ad}(\cO_E)) \otimes \bF
\times
H^1(G_K, U^{\ad}(\cO_E)) \otimes \bF
\to H^2(G_K, Z(U)(\cO_E)) \otimes \bF
$$
is non-trivial.
\end{itemize}

\vspace{3mm}

\noindent
\text{\bf Remark}
Note that $H^1_{\LD}(U^{\ad}(\cO_E)) \cong H^1(G_{K'}, U^{\ad}(\cO_E))$, and
$H^1(G_{K'}, U^{\ad}(\cO_E))^{G_K}=H^1(G_{K}, U^{\ad}(\cO_E))$.
The symmetric pairing in the theorem
is the restriction to $H^1(G_{K'}, U^{\ad}(\cO_E))$
of the symmetric pairing
defined in Lemma \ref{lem:h1-cup}.

\begin{proof}
Assume $H^2(G_K, Z(U)(\bF)) \ne 0$.
Consider the diagram
$$
\xymatrix{
H^1(G_K, U^{\ad}(\cO_E))
\times
H^1(G_K, U^{\ad}(\cO_E))
\ar[r]\ar@{^{(}->}[d] &
H^2(G_K, Z(U)(\cO_E))\ar[d]^{\cong}
\\
H^1(G_{K'}, U^{\ad}(\cO_E))
\times
H^1(G_{K'}, U^{\ad}(\cO_E))
\ar[r] &
H^2(G_{K'}, Z(U)(\cO_E))
\\
Z^1_{\LD}(U^{\ad}(\cO_E))
\times
Z^1_{\LD}(U^{\ad}(\cO_E))\ar[u]
\ar[r] &
C^2(Z(U)(\cO_E))\ar[u]
}
$$
By Lemma \ref{lem:char-p-ker-estimate}, the kernel of
$$
H^1(G_{K'}, U^{\ad}(\cO_E))/\varpi
\times
H^1(G_{K'}, U^{\ad}(\cO_E))/\varpi
\to
H^2(G_{K'}, Z(U)(\bF))
$$
has $\bF$-dimension at most $\dim_E U^{\ad}(E)$.
Write $\Delta$ for $G_{K}/G_{K'}$, which acts on $H^1(G_{K'}, U^{\ad}(\cO_E))$
with fixed-point subspace $H^1(G_{K}, U^{\ad}(\cO_E))$.

By an averaging argument (explained below), the kernel
of
$$
H^1(G_{K}, U^{\ad}(\cO_E))/\varpi
\times
H^1(G_{K}, U^{\ad}(\cO_E))/\varpi
\to
H^2(G_{K}, Z(U)(\bF))
$$
is contained in the kernel of
$$
H^1(G_{K'}, U^{\ad}(\cO_E))/\varpi
\times
H^1(G_{K'}, U^{\ad}(\cO_E))/\varpi
\to
H^2(G_{K'}, Z(U)(\bF))
$$
and thus has $\bF$-dimension at most $\dim_E U^{\ad}(E)$.
(Let $[c]\in H^1(G_{K}, U^{\ad}(\cO_E))/\varpi$ and suppose
$[c]\cup [d]=0$ for all $[d] \in H^1(G_{K}, U^{\ad}(\cO_E))/\varpi$.
Let $[c']\in H^1(G_{K'}, U^{\ad}(\cO_E))/\varpi$.
Then $\sum_{\sigma\in \Delta}\sigma([c]\cup [c'])= [c]\cup \sum_{\sigma\in \Delta}[c'] = 0$.
Since $H^2(G_K, Z(U)(\bF)) \ne 0$, we have
$H^2(G_K, Z(U)(\bF)) = H^2(G_{K'}, Z(U)(\bF))$ and thus
$\sum_{\sigma\in \Delta}\sigma([c]\cup [c']) = \#\Delta \sigma([c]\cup [c'])$.)

We remark that as a finitely generated module over a DVR,
$H^1(G_K, U^{\ad}(\cO_E))$ is the direct sum of its torsion-free part and its torsion part;
and $H^1(G_K, U^{\ad}(E))=H^1(G_K, U^{\ad}(\cO_E))_{\text{torson-free}}\otimes_{\cO_E}E$.

By the local Euler characteristic,
\begin{eqnarray}
\dim_E H^1(G_K, U^{\ad}(E))
&=& 
\dim_E H^2(G_K, U^{\ad}(E))
+\dim_E H^0(G_K, U^{\ad}(E))
+ \dim_EU^{\ad}(E)[K:\Qp]\nonumber\\
&\ge& 
\dim_E H^2(G_K, U^{\ad}(E))
+ \dim_EU^{\ad}(E).\nonumber
\end{eqnarray}
We will now consider two possibilities: $H^2(G_K, U^{\ad}(\bF)) \ne 0$ and $H^2(G_K, U^{\ad}(\bF)) = 0$.

{\bf Case $H^2(G_K, U^{\ad}(\bF)) \ne 0$.} \hspace{5mm}
Since $H^2(G_K, U^{\ad}(\bF))\ne 0$,
$H^2(G_K, U^{\ad}(\cO_E))$ is non-trivial.
So either we have
$\dim_E H^2(G_K, U^{\ad}(E)) > 0$,
or $H^2(G_K, U^{\ad}(\cO_E))$
has non-trivial torsion.
If $H^2(G_K, U^{\ad}(\cO_E))$
has non-trivial torsion, then
again by the local Euler characteristic (mod $\varpi$ version),
$H^1(G_K, U^{\ad}(\cO_E))$
also has non-trivial torsion.
In either case,
$\dim_{\bF}H^1(G_K, U^{\ad}(\cO_E))/\varpi\ge \dim_EU^{\ad}(E)+1$.
So the kernel of the cup product is a proper subspace of
$H^1(G_K, U^{\ad}(\cO_E))/\varpi$.

{\bf Case $H^2(G_K, U^{\ad}(\bF)) = 0$.} \hspace{5mm}
By Nakayama's Lemma, 
$H^2(G_K, U^{\ad}(\cO_E))=0$.
By \cite{EG23},
there exists a perfect $\cO_E$-complex
$[C^0\to C^1\to C^2]$ concentrated in degrees $[0,2]$ which computes $H^\bullet(G_K, U^{\ad}(\cO_E))$.
By the universal coefficient theorem,
there exists a short exact sequence
$$
0\to
H^1(C^\bullet)\otimes \bF \to H^1(C^\bullet \otimes \bF)
\to Tor^{\cO_E}_1(H^2(C^\bullet), \bF)
\to 0
$$
So
$H^1(G_K, U^{\ad}(\cO_E))\otimes_{\cO_E}\bF=H^1(G_K, U^{\ad}(\bF))$.
We assume (i) and (ii) are false, and try to get a contradiction.
The kernel of 
$$
H^1(G_K, U^{\ad}(\cO_E)) \otimes \bF
\times
H^1(G_K, U^{\ad}(\cO_E)) \otimes \bF
\to H^2(G_K, Z(U)(\cO_E)) \otimes \bF
$$ has dimension $h^1:=\dim_{\bF} H^1(G_K, U^{\ad}(\bF))$.
By the local Euler characteristic, 
$$\text{(*)} \hspace{5mm}h^1=\dim_EU^{\ad}(E)[K:\Qp] + \dim_{\bF} H^0(G_K, U^{\ad}(\bF)).$$
By
Lemma \ref{lem:char-p-ker-estimate-variant},
the kernel $k_Z$ of 
$$
Z^1_{\LD}(U^{\ad}(\bF))
\times
Z^1_{\LD}(U^{\ad}(\bF))
\to H^2(G_{K'}, Z(U)(\bF))
$$
has
dimension at most $\dim_EU^{\ad}(E)$.
Since the cup product is trivial on $H^1(G_K, U^{\ad}(\bF))$,
we have
$$\text{(**)} \hspace{5mm}\dim k_Z \ge \dim H^1(G_K, U^{\ad}(\bF)) + \dim B^1_{\LD}(U^{\ad}(\bF))= h^1 +  \dim B^1_{\LD}(U^{\ad}(\bF)).$$
Combining ($*$) and ($**$),
we have
$$
\dim_E U^{\ad}(E) \ge \dim_{\bF} k_Z \ge \dim_E U^{\ad}(E) [K:\Qp] + \dim_{\bF}H^0(G_K, U^{\ad}(\bF))+  \dim B^1_{\LD}(U^{\ad}(\bF))
$$
So we conclude that
\begin{align*}
1 &= [K:\Qp] \\
0 &= H^0(G_K, U^{\ad}(\bF))\\
0 &= B^1_{\LD}(U^{\ad}(\bF))
\end{align*}
In particular, we have $H^0(G_{K'}, U^{\ad}(\bF))=U^{\ad}(\bF)$, and the kernel of the cup product
on $H^1(G_{K'}, U^{\ad}(\bF))$
has dimension exactly $\dim_E U^{\ad}(E)$.
However, by Example \ref{example:trivial}, the cup product 
on $H^1(G_{K'}, U^{\ad}(\bF))$
is non-degenerate by local Tate duality.
\end{proof}

Theorem \ref{thm:nontrivial} is used in the following scenerio.

\vspace{3mm}

\paragraph{Lemma}
\label{lem:K-primed}
Let $L$ be a split reductive group over $\bF$.
Let $r:G_K\to L(\bF)$ be a Galois representation valued in $L$.
Let $r^{ss}$ be the semi-simplification of $r$.
Write $G_{K'}$ for the kernel of $r^{ss}$.
Then the degree $[K':K]$ divides $(q-1)^r \# W_L$
where 
\begin{itemize}
\item $r$ is the rank of $L$,
\item $q$ is a power of $p$, and
\item $\# W_L$ is the cardinality of the Weyl group of $L$.
\end{itemize}

\begin{proof}
By \cite{L22}, $r^{ss}$ is tamely ramified and factors through the
normalizer of a maximal torus of $L$ (after possibly extending the base field).
\end{proof}

In particular, if $L=G_2$ and $p>3$,
the kernel of $r^{ss}$ defines a Galois extension $K'/K$
of prime-to-$p$ degree;
and $r|_{G_{K'}}$ is Lyndon-Demu\v{s}kin since it has trivial semi-simplification.

\vspace{3mm}

\section{\small Non-abelian obstruction theory via Lyndon-Demu\v{s}kin cocycle group with external Galois action}
\label{sec:manipulate}

Let $K/\Qp$ be a $p$-adic field.
Let $E/\Qp$ be a finite extension
with ring of integers $\cO_E$,
residue field $\bF$,
and uniformizer $\varpi$.

Let $L$ be a split reductive group over $\cO_E$.
Fix a Galois representation
$$
r^{\circ}: G_K \to L(\cO_E)
$$
throughout this section.

Let $U$ be a unipotent group over $\cO_E$
whose adjoint group is abelian.
Let $Z(U)$ be the center of $U$.
The adjoint group $U^{\ad}$
is defined to be $U/Z(U)$.

Fix a group scheme homomorphism
$\phi:L\to \Aut(U)$ throughout this section.
In particular, there is a Galois action
$\phi(r^{\circ}): G_K\xrightarrow{r^{\circ}} L(\cO_E)
\xrightarrow{\phi(\cO_E)} \Aut(U)(\cO_E)$.
We will talk about non-abelian Galois cohomology
$H^{\bullet}(G_K, U(\cO_E))$
and
$H^{\bullet}(G_K, U(\bF))$
using this Galois action throughout this
section.

Let $K'/K$ be a prime-to-$p$,
finite Galois extension of $K$ containing
the group of $p$-th root of unity,
such that
$r^{\circ}(G_{K'})\subset L(\cO_E)$
is a pro-$p$ group.
Write $\Delta$ for $\Gal(K'/K)$.
Set $\Gamma := G_K$, and $H := G_{K'}$.

\subsection{Non-abelian inflation-restriction}
\vspace{3mm}

\paragraph{Non-abelian Galois cohomology}
We recall a few facts about the non-abelian version
of Galois cohomology.
Let
$$
0\to A\to B\to C\to 0
$$
be a short exact sequence of groups with
continuous $\Gamma$-action.
If $A\to B$ is \textit{central}, that is, $A$ is contained in the center of $B$,
then we have a long exact sequence of pointed sets (\cite[Proposition 43, 5.7]{Se02})
\begin{eqnarray}
\begin{split}
1 &\to A^\Gamma \to B^\Gamma \to C^\Gamma \\
& \xrightarrow{} H^1(\Gamma, A)
\to H^1(\Gamma, B) \to H^1(\Gamma, C) \\
& \xrightarrow{\delta} H^2(\Gamma, A)
\end{split}\nonumber
\end{eqnarray}
Let $H \subset \Gamma$ be a closed normal subgroup.
Then there is an exact sequence (\cite[5.8]{Se02})
\begin{equation}
\label{eq:Se0258}
1\to H^1(\Gamma/H, A^H) 
\to H^1(\Gamma, A) \to H^1(H, A)^{\Gamma/H}.
\end{equation}
If $A$ is an abelian group, then the sequence above can be upgraded
to the inflation-restriction exact sequence:
$$
1\to H^1(\Gamma/H, A^H) 
\to H^1(\Gamma, A) \to H^1(H, A)^{\Gamma/H} \to H^2(\Gamma, A^H).
$$

\vspace{3mm}
\paragraph{Theorem} \cite[Theorem 3.15]{Ko02}
Let $\Gamma$ be a profinite group,
$H$ a normal subgroup of finite index,
and $A$ an (abelian) $G$-module whose elements have finite order
coprime to $(\Gamma:H)$. Then
$$
H^n(\Gamma/H, A^H) = 0
$$
for all $n\ge 1$, and the restriction
$$
H^n(\Gamma, A)\to H^n(H, A)^{\Gamma/H}
$$
is an isomorphism.

\vspace{3mm}

Let $R$ be either $\cO_E$ or $\bF$.
For ease of notation, write
$U$ for $U(R)$ in this paragraph.
The fact above implies the following diagram commutes, with exact columns
$$
\xymatrix{
    H^1(\Gamma, Z(U)) \ar[r]^{\cong}_{\res} \ar[d] &
    H^1(H, Z(U))^{\Delta} \ar[d] \\
    H^1(\Gamma, U) \ar@{^{(}->}[r]_{\res} \ar[d]^{\alpha_1} &
    H^1(H, U)^{\Delta} \ar[d]^{\alpha_2} \\
    H^1(\Gamma, U^{\ad}) 
    \ar[r]^{\cong}_{\res} \ar[d]^{\delta_1} &
    H^1(H, U^{\ad})^{\Delta} \ar[d]^{\delta_2} \\
    H^2(\Gamma, Z(U)) \ar@{^{(}->}[r] &
    H^2(H, Z(U))
}
$$
The injectivity of the second line follows from Equation (\ref{eq:Se0258}).

\vspace{3mm}
\paragraph{Proposition}
\label{prop:inf-res}
The restriction map of non-abelian $1$-cocycles
$$
H^1(\Gamma, U) \to H^1(H, U)^{\Delta}
$$
is a bijection.

\begin{proof}
It follows from diagram chasing:
Let $[c] \in H^1(H, U)^{\Delta}$.
Since $\delta_1(\res^{-1}(\alpha_2[c])) = \delta_2(\alpha_2[c]) = 0$,
there exists $[b]\in H^1(\Gamma, U)$ such that
$\alpha_1(\res([b])) = \alpha_2([c])$.
Since $\alpha_2^{-1}(\alpha_2([c]))$ is a $H^1(H, Z(U))^{\Delta}$-torsor,
we can twist $[b]$ to make $\res([b]) = [c]$.
\end{proof}

\vspace{3mm}

\paragraph{Representation-theoretic interpretation of non-abelian $1$-cocycles}
\label{par:non-ab-h1}
Let $\fP$ be a group which is a semi-direct product $\fL\ltimes \fU$.
Let $q_{\fL}: \fP\to \fL$ be the quotient map.
Fix a section $\fL \to \fP$ of $q_{\fL}$, which allows us to identify
(set-theoretically) $\fP$ with $\fU \times \fL$;
and write $q_{\fU}:\fP\to \fU$ be the projection map.
For $g\in \fP$, write $g = g_{\fU}g_{\fL}$ such that $g_{\fU}\in \fU\times \{1\}$
and $g_L\in \{1\}\times \fL$.
Let $\bar\tau: \Gamma \to \fL$ be a group homomorphism.
Let $\tau: \Gamma \to \fP$ be a lifting of $\bar\tau$.
Set $c:= q_{\fU}\circ \tau: \Gamma\to \fU$.
Then
\begin{eqnarray}
c(gh) &=& q_{\fU}(\tau(g)\tau(h))=q_{\fU}(\tau(g)_{\fU}\tau(g)_{\fL}\tau(h)_{\fU}\tau(h)_{\fL})\nonumber\\
&=&q_{\fU}(\tau(g)_{\fU}\tau(g)_{\fL}\tau(h)_{\fU}\tau(g)_{\fL}^{-1}\tau(gh)_{\fL})\nonumber\\
&=&c(g) (\tau(g)_{\fL} c(h)\tau(g)^{-1}_{\fL})\nonumber\\
&=:&c(g) (\tau(g)_{\fL}\cdot c(h))\nonumber
\end{eqnarray}
is a (non-abelian) crossed homomorphism.
Two liftings $\tau_1$ and $\tau_2$ are equivalent if
there exists an element $n\in \fU$ such that $\tau_1=n \tau_2 n^{-1}$.
So $H^1(\Gamma, \fU)$ classifies liftings $\tau$ of $\bar\tau$ up to equivalence.

\vspace{3mm}

\subsubsection{Lifting characteristic $p$ cocycles via
inflation-restriction}
~

Let $[\bar c] \in H^1(\Gamma, U(\bF))$
be a characteristic $p$ cocycle.
Assume the restriction $[\bar c|_H]
\in H^1(H, U(\bF))$
has a characteristic $0$
lift $[c_h]\in H^1(H, U(\cO_E))$.
We want to build a lift $[c]\in H^1(\Gamma, U(\cO_E))$ of $[\bar c]$
using 
$[c_h]$.

Note that when $U$ is an abelian group,
this can be easily achieved by
taking the average
$$
[c] := \frac{1}{\# \Delta}\sum_{g\in \Delta}g\cdot [c_h].
$$
Here we identify $H^1(\Gamma, U(\cO_E))$
with a subset of $H^1(H, U(\cO_E))$
via Proposition 
\ref{prop:inf-res}.

Such a trick does not work anymore when
$U$ is non-abelian. Nonetheless,
we have the following:

\vspace{3mm}

\paragraph{Lemma}
\label{lem:non-ab-inf-res-lift}
If 
there exists $[c_h]\in H^1(H, U(\cO_E))$
and $[d]\in H^1(\Gamma, U^{\ad}(\cO_E))$
such that
\begin{itemize}
\item 
$\alpha_2([c_h])=\res ([d])$ and
\item
$[c_h]$ mod $\varpi=[\bar c|_H]$
\end{itemize}
 then there exists
$[c]\in H^1(\Gamma, U(\cO_E))$
which is a lifting of $[\bar c]$.

$$
\xymatrix{
    &H^1(\Gamma, Z(U)(\cO_E)) \ar@{^{(}->}[r]_{\res} \ar[d] &
    H^1(H, Z(U)(\cO_E)) \ar[d] \\
    &H^1(\Gamma, U(\cO_E)) \ar@{^{(}->}[r]_{\res} \ar[d]^{\alpha_1} &
    H^1(H, U(\cO_E))\ar[d]^{\alpha_2}  
& \hspace{-10mm} \ni [c_h]
    \\
    [d]\in \hspace{-10mm}&H^1(\Gamma, U^{\ad}(\cO_E)) 
    \ar[r]_{\res} \ar[d]^{\delta_1} &
    H^1(H, U^{\ad}(\cO_E)) \ar[d]^{\delta_2} \\
    &H^2(\Gamma, Z(U)(\cO_E)) \ar@{^{(}->}[r] &
    H^2(H, Z(U)(\cO_E))
}
$$

\begin{proof}
Since
$$
\delta_1([d]) = \delta_2(\alpha_2([c_h]))=0,
$$
$[d] = \alpha_1([c'])$ for some
$[c']\in H^1(\Gamma, U(\cO_E))$.
Since $\res([c'])$ and $[c_h]\in H^1(H, U(\cO_E))$
have the same image in $H^1(H,U^{\ad}(\cO_E))$ (via $\alpha_2$),
it makes sense to talk about the difference
$\res([c'])-[c_h] \in H^1(H, Z(U)(\cO_E))$.
\footnote{$H^1(H, U(\cO_E))$
is a $H^1(H, Z(U)(\cO_E))$-principle
homogeneous space.
}
Consider the following diagram
$$
\xymatrix{
H^1(\Gamma, Z(U)(\cO_E))
\ar[r] \ar@{^{(}->}[d]^{\res}&
H^1(\Gamma, Z(U)(\bF))
\ar[r]^{\delta} \ar@{^{(}->}[d]^{\res}&
H^2(\Gamma, Z(U)(\cO_E)) \ar@{^{(}->}[d]\\
H^1(H, Z(U)(\cO_E))
\ar[r] &
H^1(H, Z(U)(\bF))
\ar[r]^{\delta} &
H^2(H, Z(U)(\cO_E))
}
$$
Let $[\bar c']\in H^1(\Gamma, Z(U)(\bF))$
be the reduction mod $\varpi$ of $[c']$.
Since $\res([\bar c'])-[\bar c_h]$
has a lift, $$\delta(\res([\bar c']-[\bar c_h]))=0\in H^2(H, Z(U)(\bF))$$
by the exactness of the second row of the diagram above.
Therefore
$$\delta([\bar c']-[\bar c])
= \delta(\res([\bar c']-[\bar c]))
= \delta(\res([\bar c']-[\bar c_h]))=0$$
and $[\bar c']-[\bar c]\in H^1(\Gamma, Z(U)(\bF))$
has a characteristic $0$ lift $[x]$,
and $[c]:=[c']-[x]$ is a lift of $[\bar c]$.
\end{proof}

The purpose of the whole Section \ref{sec:manipulate}
is to prove Theorem \ref{thm:non-abelian-obstruction}, which 
extends the above lemma.

\subsection{External Galois action on the
Lyndon-Demu\v{s}kin cocycle group}
\label{ss:external-Galois}
~

\vspace{3mm}

The earlier subsection shows there is
an identification
$$
H^1(\Gamma, U(\cO_E)) \cong H^1(H, U(\cO_E))^{\Delta}.
$$
The goal of this subsection is to 
upgrade this identification
to the cochain level.

Since the Galois action
$$
\phi(r^{\circ})|_{G_{K'}}:G_{K'} \to
U(\cO_E)
$$
is Lyndon-Demu\v{s}kin,
we have a Lyndon-Demu\v{s}kin
complex $C^{\bullet}_{\LD}(U(\cO_E))$
computing $H^{\bullet}(H, U(\cO_E))$.
Recall (\ref{sss:nilcoef}) that a $1$-cochain
$c\in C^1_{\LD}(U(\cO_E))$
is the same as a function
$$
\mathfrak{c}: \langle x_0, \cdots, x_{n+1}\rangle \to (\Lie U)(\cO_E)
$$
such that
$$
\mathfrak{c}(gh) = \mathfrak{c}(g) + 
g\cdot \mathfrak{c}(h) + \frac{1}{2}[
\mathfrak{c}(g), g\cdot\mathfrak{c}(h)]
$$
for all $g,h$; or, equivalently, a function
$$
c: \langle x_0, \cdots, x_{n+1}\rangle \to  U(\cO_E)
$$
such that
$$
    c(gh) = c(g)(g\cdot c(h))
$$
for all $g,h$.

A cochain $c:\langle x_0, \cdots, x_{n+1}\rangle\to U(\cO_E)$
lies in
$Z^1_{\LD}(U(\cO_E))$
if and only if it factors through
    the (discrete) Demu\v{s}kin group
    $\langle x_0, \cdots, x_{n+1} | R\rangle$
    (see the proof of Proposition \ref{prop:HLDvsH}).

Let $c\in Z^1_{\LD}(\cO_E)$,
regarded as a function 
$\langle x_0, \dots, x_{n+1}|R\rangle
\to U(\cO_E)$.
Since $U(\cO_E)$ is a pro-$p$ group,
the crossed homomorphism
necessarily
factors through the pro-$p$ completion,
that is, we have a commutative diagram
$$
\xymatrix{
    &\langle x_0, \dots, x_{n+1}|R
    \rangle \ar[d]^{\pi}\ar[r]^{\hspace{5mm}c}
    & U(\cO_E) \\
    G_{K'}(p)\ar@{=}[r] &\widehat{
    \langle x_0, \dots, x_{n+1}|R\rangle
    }^p\ar[ru]^{\widehat{c}}
}
$$

Since we have identified
the pro-$p$ quotient of
    $G_{K'}$ with the pro-$p$ completion of $\langle x_0, \cdots, x_{n+1} | R\rangle$,
 we can define,
 for each $g\in G_K$, an
     automorphism $\alpha_g$ of $Z^1_{\LD}(U(\cO_E))$ via
$$
\alpha_g(c) := (h\mapsto g\cdot \widehat{c}(g^{-1}\pi(h)g)).
$$
So we defined an action of $G_K$
on $Z^1_{\LD}(U(\cO_E))$.

For ease of notation, write
$g\cdot c$ for $\alpha_g(c)$.
Note that
$(g\cdot c)(h) = (\alpha_g(c))(h)$
is different from $g\cdot c(h)$.
We apologize for the confusing notation.

\vspace{3mm}
\paragraph{Remark}
We don't know whether or not
we can define a $G_K$-action
on the whole cochain group $C^1_{\LD}(U(\cO_E))$.
It seems to involve some subtle 
combinatorial group theory.

\vspace{3mm}

\paragraph{Digression} It is curious to know if
the cup product
$$
\cup:Z^1_{\LD}(U^{\ad}(\cO_E)) \times
Z^1_{\LD}(U^{\ad}(\cO_E)) \to C^2_{\LD}(Z(U(\cO_E)))
$$
is compatible with the $G_K$-action.

\vspace{3mm}
This answer would be affirmative if, for
example, for each $g\in G_K$,
the conjugation by $g$
$$
\phi_g : G_{K'} \to G_{K'}
$$
can be lifted to an automorphism of
free pro-$p$ groups on $(n+2)$-generators
$$
\phi_g:\langle x_0, \cdots, x_{n+1}\rangle
\to
\langle x_0, \cdots, x_{n+1}\rangle.
$$
This is closely related to the so-called \textit{Dehn-Nielsen} theorem.
Classically, Dehn-Nielsen is saying all automorphism of the fundamental group
of the genus $g$ closed surface $M_g$ are induced by a homeomorphism.
The algebraic version of Dehn-Nielsen can be formualted as,
under the usual presentation of $F=\langle a_1,b_1,\cdots,a_g,b_g\rangle\to\langle a_1,b_1,\cdots,a_g,b_g|[a_1,b_1]\cdots[a_g,b_g]\rangle\cong\pi_1(M_g)$,
all automorphism of $\pi_1(M_g)$ are induced from an automorphism of the free group $F$.

\vspace{3mm}

\text{\bf Conjecture} (Pro-$p$ Dehn-Nielsen)
All automorphisms of the pro-$p$ completion of $\langle x_0,\cdots,x_{n+1}|R\rangle$ are induced by an automorphism
of the pro-$p$ completion of $\langle x_0,\cdots,x_{n+1}\rangle$.


\vspace{3mm}

\subsection{Constructing non-abelian cocycles}~

Recall that
$
H^1(H, U^{\ad})^{\Delta} = H^1(G_K, U^{\ad})$
where $H=G_{K'}$ and $K'/K$ is a normal extension of prime-to-$p$ degree.
Define
\begin{eqnarray}
(Z^1_{LD})^{\Delta}
&:=& \{x\in Z^1_{\LD}| \text{image of }x
\text{ in }H^1\text{ is contained
in }(H^1)^{\Delta}\}\nonumber\\
&=& \{x\in Z^1_{\LD}|g\cdot x-x\in B^1_{\LD}\text{ for all }g\in G_K\}\nonumber
\end{eqnarray}
Since $Z^1_{\LD}(U^{\ad}(\cO_E))^{\Delta}$
is a submodule of
a finite flat $\cO_E$-module,
it is finite $\cO_E$-flat.

\vspace{3mm}

We keep all notations from the previous subsections.

Assume $Z(U)(\cO_E) = \cO_E$
from now on.

We fix some notation.
The quotient
$U\to U/Z(U)=U^{\ad}$
induces maps
$\ad: Z^1_{\LD}(U(\cO_E))\to Z^1_{\LD}(U^{\ad}(\cO_E))$.

\vspace{3mm}

\subsubsection{Lemma}
\label{lem:cup-0-lift}
Assume 
$p\ne 2$ and
the cup product
$$
\text{($\dagger$)}\hspace{10mm} \cup:
H^1(G_K, U^{\ad}(\cO_E)) \otimes\bF
\times
H^1(G_K, U^{\ad}(\cO_E)) \otimes \bF
\to H^2(G_K, Z(U)(\bF))
$$
is non-trivial.

Let $(\bar c, \bar f)\in Z^1_{\LD}(U(\bF))$
(using Lemma \ref{lem:non-abelian-h1}).
Assume $\bar c\in Z^1_{\LD}(U^{\ad}(\bF))^{\Delta}$. 
If $\bar c$ admits a characteristic $0$
lift $c'\in Z^1_{\LD}(U^{\ad})(\cO_E)$,
then $(\bar c, \bar f)$ admits
a lift $(c, f)\in Z^1_{\LD}(U(\bZp))$
such that $c\in Z^1_{\LD}(U^{\ad}(\bZp))^{\Delta}$.

\begin{proof}
Pick an arbitrary lift $f\in C^1_{\LD}(Z(U)(\cO_E))$ of $\bar f$.
Choose a system of representatives
$\{g_i\}\subset G_K$ of $\Delta$.
By replacing $c'$ by the $\Delta$-average $\frac{1}{\#\Delta}\sum g_i\cdot c'$ + some coboundary
(which is also a lift of $[\bar c]$),
we assume $c'\in Z^1_{\LD}(U^{\ad}(\bZp))^{\Delta}$.

Let $\lambda\in \bZp^{\times}$ be a scalar.

Since the symmetric bilinear pairing ($\dagger$) is
non-trivial, there exists
$y\in Z^1_{\LD}(U^{\ad}(\cO_E))^{\Delta}$ such that
$y\cup y\ne 0$ mod $\varpi$.
Consider
$$
(c'+ \lambda y)\cup (c'+\lambda y) + d^2(f)
=c'\cup c' + d^2(f) + 2 \lambda c'\cup y + \lambda^2 y\cup y
\in C^2(Z(U)(\cO_E))\cong \cO_E
$$
which is a degree two polynomial in $\lambda$
whose Newton polygon has vertices
$(0,+),(1,+\text{~or~}0),(2,0)$
and thus has at least one solution $\lambda_0$ with positive $p$-adic valuation;
here ``$+$'' means a positive number.
Set $(c,f):=(c'+\lambda_0 y, f)$. 


We have $(c,f)\in Z^1_{\LD}(U(\bZp))$ by Lemma \ref{lem:non-abelian-h1} and $c\in Z^1_{\LD}(U^{\ad}(\bZp))^{\Delta}$.
\end{proof}

\vspace{3mm}

\subsubsection{Theorem}
\label{thm:non-abelian-obstruction}
Assume  $p\ne 2$ and
the cup product
$$
\cup:
H^1(G_K, U^{\ad}(\cO_E))\otimes \bF
\times
H^1(G_K, U^{\ad}(\cO_E))\otimes \bF
\to H^2(G_K, Z(U)(\bF))
$$
is non-trivial.

Let $[(\bar c, \bar f)]\in H^1(G_K,U(\bF))$ be a characteristic $p$ cocycle.
If
$[\bar c|_{G_{K'}}]\in H^1(G_{K'}, U^{\ad}(\bF))$ admits
a characteristic $0$ lift in $H^1(G_{K'}, U^{\ad}(\bZp))$,
then $[(\bar c, \bar f)]$
 admits a characteristic $0$
 lift 
 $[(c, f)]\in H^1(G_K, U(\bZp))$.

\begin{proof}
We choose a cocycle
$(\bar c, \bar f)\in Z^1_{\LD}(U(\bF))$ which defines the cohomology class
$[(\bar c, \bar f)]$.
Clearly $\bar c\in Z^1_{\LD}(U^{\ad}(\bF))^{\Delta}$.
Say $[d]\in H^1(G_{K'}, U^{\ad}(\bZp))$
is a lift of $[\bar c]$,
which is defined  by $d\in Z^1_{\LD}(U^{\ad}(\bZp))$.
Write $\bar d$ for the image of $d$
in  $Z^1_{\LD}(U^{\ad}(\bFp))$.
By changing $d$ by a coboundary,
we can assume $\bar d = \bar c$.

Lemma \ref{lem:cup-0-lift}
produces $(c, f)\in Z^1_{\LD}(U(\bZp))$
such that $c\in Z^1_{\LD}(U^{\ad}(\bZp))^{\Delta}$.
Now the theorem follows from
Lemma \ref{lem:non-ab-inf-res-lift}.
\end{proof}

Theorem \ref{thm:non-abelian-obstruction}
is saying that when $U$ is a unipotent group of class $2$
with $1$-dimensional center,
there exists a short exact sequence of
pointed sets
$$
H^1(G_K, U(\bZp))
\to H^1(G_K, U(\bFp))
\xrightarrow{\delta} H^2(G_{K'}, U^{\ad}(\bZp))
$$
under technical assumptions.

Combining Theorem \ref{thm:non-abelian-obstruction}
and Theorem \ref{thm:nontrivial},
we have very nice obstruction theory 
for lifting mod $\varpi$ cohomology classes
in the mildly regular case.

\vspace{3mm}
\subsubsection{Theorem}
\label{thm:obstruction-theory}
Assume $p\ne 2$
and 
$Z(U)(\cO_E) = \cO_E$.
Let $r :G_K \to L(\cO_E)$ be a fixed continuous group homomorphism
and equip $U(\bZp)$ with the $G_K$-action $G_K\xrightarrow{r} L(\bZ_p)\to \Aut(U(\bZp))$.
Let $K'/K$ be a finite Galois extension of 
prime-to-$p$ degree
such that
$r|_{G_{K'}}$ is Lyndon-Demu\v{s}kin and mildly regular.

There is a short exact sequence of pointed sets
$$
H^1(G_K, U(\bZp))
\to H^1(G_K, U(\bFp))
\xrightarrow{\delta} H^2(G_{K'}, U^{\ad}(\bZp))
$$
where $\delta$ has a factorization
$H^1(G_K, U(\bFp))\xrightarrow{z}
H^1(G_{K}, U^{\ad}(\bFp)) \to
H^2(G_{K'}, U^{\ad}(\bZp))$.

\begin{proof}
Write $\Delta$ for $G_K/G_{K'}$.
By the moreover part of Theorem \ref{thm:nontrivial}, there are two cases to consider.

\text{\bf Case I}: the cup product
($\dagger$) $H^1(G_K, U^{\ad}(\bZp)) \otimes \bF
\times
H^1(G_K, U^{\ad}(\bZp)) \otimes \bF
\to H^2(G_K, Z(U)(\bZp)) \otimes \bF$
is non-trivial.
This is a corollary of Theorem \ref{thm:non-abelian-obstruction}.

\text{\bf Case II}: $H^2(G_K, Z(U)(\bF))=0$.
The short exact sequence $0\to Z(U)(\cO_E) \to Z(U)(\cO_E) \to Z(U)(\bF)\to 0$
induces a long exact sequence $H^2(G_K, Z(U)(\cO_E))\to H^2(G_K, Z(U)(\bF))\to 0$.
By Nakayama's lemma, $H^2(G_K, Z(U)(\cO_E))=0$,
and thus $H^2(G_K, Z(U)(\bZp))=0$ by flat base change.

Let $[(\bar c,\bar f)]\in H^1(G_K, U(\bFp))$
be a cohomology class defined by
$(\bar c, \bar f)\in Z^1_{\LD}(U(\bFp))$.

Set $\delta:H^1(G_K, U(\bFp))
\to H^2(G_{K'}, U^{\ad}(\bZp))$
to be the composite
$$
H^1(G_K, U(\bFp)) \xrightarrow{[(\bar c,\bar f)]\mapsto [\bar c]}
H^1(G_K, U^{\ad}(\bFp))
\to H^2(G_{K'}, U^{\ad}(\bZp)).
$$

If $\delta([(\bar c, \bar f)])=0$,
then there exists a lift $c \in Z^1_{\LD}(U^{\ad}(\bZp))$ of $\bar c$.
By replacing $c$ by the $\Delta$-average of $c$,
we assume $c  \in Z^1_{\LD}(U^{\ad}(\bZp))^\Delta$.
Since $H^2(G_K, Z(U)(\bZp))=0$,
$[c\cup c] = 0$ and thus
there exists $g\in C^1_{\LD}(Z(U)(\bZp))^\Delta$
such that
$c \cup c = - d^2(g)$.
Write $\bar g$ for the image of $g$
in $C^1_{\LD}(Z(U)(\bFp))$.
We have $\bar g - \bar f \in Z^1_{\LD}(Z(U)(\bFp))^\Delta$.
Since $H^2(G_K, Z(U)(\bZp))=0$,
there exists a lift $h\in Z^1_{\LD}(Z(U)(\bZp))^\Delta$
of $\bar f - \bar g$.
It is clear that $[(c, g+h)]\in H^1(G_K, U(\bZp))$
is a lift of $[(\bar c, \bar f)]$.
\end{proof}

\subsubsection{Corollary}
\label{cor:obstruction-theory}
Assume $p\ne 2$
and 
$Z(U)(\cO_E) = \cO_E$.
Let $r :G_K \to L(\cO_E)$ be a continuous group homomorphism.

If there exists a
finite Galois extension $K'/K$ of prime-to-$p$ degree
such that
$r|_{G_{K'}}$ is Lyndon-Demu\v{s}kin and mildly regular, then
there is a short exact sequence of pointed sets
$$
H^1(G_K, U(\bZp))
\to H^1(G_K, U(\bFp))
\xrightarrow{\delta} H^2(G_{K}, U^{\ad}(\bZp))
$$
where $\delta$ has a factorization
$H^1(G_K, U(\bFp))\xrightarrow{z}
H^1(G_{K}, U^{\ad}(\bFp)) \to
H^2(G_{K}, U^{\ad}(\bZp))$.

\begin{proof}
It is an immediate consequence of Theorem \ref{thm:obstruction-theory}.
\end{proof}
\vspace{3mm}

\section{The Machinery for lifting non-abelian cocycles}
\label{sec:MACHINE}

Let $K/\Qp$ be a p-adic field.
Let $E/\Qp$ be the coefficient field
with ring of integers $\cO_E$, residue field
$\bF$ and uniformizer $\varpi$.

\subsubsection{Emerton-Gee stacks}
Let $H$ be a connected reductive group over $K$
which splits over a tame extension $K_H/K$.
Denote by $\lsup LH$
the Langlands dual group $\wh H\rtimes\Gal(K_H/H)$
where $\wh H$ is the split connected reductive group over $\bZ$
whose root datum is dual to that of $H$.
The reduced Emerton-Gee stack $\cX_{\lsup LH, \red}$
is a reduced algebraic stack defined over $\bF_p$
(see \cite[Theorem 1]{L23B}).

Moreover, it is proved in many cases that $\cX_{K,\lsup LH, \red}$
is equidimensional of dimension $[K:\Qp]\dim \wh H/B_{\wh H}$
where $B_{\wh H}$ is a Borel of $\wh H$ (see \cite{L23B}).

\subsubsection{Potentially semistable lifting rings}
\label{def:crys-lifting}

Write $L:=\lsup LH$ for simplicity.
Let $\bar r : G_K \to L(\bF)$ be a mod $\varpi$
Langlands parameter, that is,
a continuous group homomorphism such that the composite
$G_K\to\lsup LH(\bFp)\to\Gal(K_H/K)$ is the canonical quotient map.
Let $\underline{\lambda}$ be a Hodge type
and let $\tau$ be a inertial Galois type (see \cite{L23D} for the definitions).
The potentially semistable deformation ring $R^{\underline{\lambda}, \tau, \mathcal{O}}_{\bar r}$
of $\bar r$ of $p$-adic Hodge type $\underline{\lambda}$
is constructed in \cite[Theorem 3.3.8]{BG19}.
It is an $\cO$-flat quotient of the universal lifting ring,
and is equidimensional of dimension $(1+\dim \wh H+[K:\Qp]\dim \wh H/B_{\wh H})$
when $\underline{\lambda}$ is a regular Hodge type.

\vspace{3mm}

\subsection{A geometric argument of
Emerton-Gee}~

\subsubsection{Definition}
\label{def:sgr}
Let $\mathcal{F}$ be a coherent sheaf
over a scheme $X=\spec R$.
We say $\mathcal{F}$ is \textit{sufficiently generically regular}
(= SGR) if for each $s \ge 1$,
the locus
$$
X_s := \{x\in \spec R|\dim \kappa(x)\otimes_R \mathcal{F} \ge s\}
$$
has codimension $\ge s+1$ in $\spec R$.

\vspace{3mm}

\subsubsection{Theorem}
\label{thm:abelian-lift}
Let $X=\spec R$ with $R$ a complete reduced, $\bZ_p$-flat local ring
that is equidimensional of dimension $(1+\dim L+\dim \cX_{L, \red})$.
Let $r^{\univ}:G_K\to L(R)$
be a family of $L$-parameters on
$X$.
Assume $X[1/p] \ne \emptyset$.
Let $F:L \to \GL(V)$ be an algebraic representation
where $V$ is a vector space scheme over $\cO_E$.

Assume $H^2(G_K, F(r^{\univ}))$ is SGR over $X$
and is supported on $X\otimes_{\bZ_p}\bF_p$.
Given any $[\bar {c}]\in H^1(G_K, F(\bar r))$,
there exists a
$\bZp$-point of $X$ giving rise
to a Galois representation
$r^{\circ}: G_K \to L(\bZp)$,
such that
the 1-cocycle $[\bar {{c}}]$
admits a lift
$[{c}] \in H^1(G_K, F(r^{\circ}))$.

\vspace{3mm}

\paragraph{Remark}
Since $H^2(G_K, -)$ (abelian coefficients) is the highest degree cohomology ($H^i(G_K, -)=0$ for $i>2$),
$H^2(G_K, -)$ commutes with base change. Thus we may view $H^2(G_K, F(r^{\univ}))$ as a coherent sheaf over $X$.

The proof is almost identical to
that of \cite[Theorem 6.3.2]{EG23}.

We would like to explain the main ideas
behind the proof, and why we need
the sufficiently generically regular
condition.

We have a complex of finitely generated
projective modules over $R$ concentrated on degree
[0, 2]
$$
C^0\to C^1\xrightarrow{d} C^2
$$
which computes the Galois cohomology
$H^{\bullet}(G_K, F(r^{\univ}))$.
Let $Z^1 := \ker (d)$ and $B^2:= \Img(d)$.
A mod $\varpi$ cocycle $[\bar c]$
is represented by an element $\bar c$ in the
kernel of $C^1/\varpi \to C^2/\varpi$.
We fix an arbitrary lift
$\widetilde{c}\in C^1$ of $\bar c$.
We can do a formal blowup
$\spec \widetilde{R}\to\spec R$,
so that the pull-back of 
$B^2$ on $\spec \widetilde{R}$
a locally free sheaf.
To make the exposition short,
we simply assume $B^2$ is locally free
over $\spec R$,
but we should not think of $\spec R$
as a local ring anymore,
because after formal blow-up, there are more closed points in the special fiber.
Now we have a sequence of locally free sheaves of modules
$$
C^1\to B^2 \to C^2.
$$
The key here is we want to regard this
as a sequence of vector bundles
instead of sheaves of modules.
Write $\mathscr{V}(\mathcal{F})$
for $\underline{\spec}(\Sym\mathcal{F}^{\vee})$, the vector bundle associated
    to the coherent sheaf $\mathcal{F}$.
So we have a sequence of scheme morphisms
$$
\begin{tikzcd}
\mathscr{V}(C^1)\arrow{r}{f}\arrow[bend left=20]{rr}{d}
&
\mathscr{V}(B^2)\arrow{r}
&
\mathscr{V}(C^2)
\\&\spec R \arrow{lu}{s}\arrow{u}{f\circ s}\arrow{ur}{d\circ s}&
\end{tikzcd}
$$
The element $\widetilde{c}$ of $C^1$ defines
a section
$s:\spec R\to \mathscr{V}(C^1)$
such that 
the section $d\circ s:\spec R\to \mathscr{V}(C^2)$
intersects with the identity section
$e_{\mathscr{V}(C^2)}:\spec R\to \mathscr{V}(C^2)$.

It turns out $\bar c\in \ker(C^1/\varpi\to C^2/\varpi)$ admits a lift
in $Z^1$, as long as the section
$f\circ s$ intersects with
the identity section $e_{\mathscr{V}(B^2)}$ of 
$\mathscr{V}(B^2)$.
The intersection $(d\circ s)\cap e_{\mathscr{V}(C^2)}$ should
occur above a codimension $1$ locus
of $\spec R$.
If the support of $H^2=C^2/B^2$ is
small (that is, has big codimension),
then the intersection should 
happen at some point $x\in\spec R$
outside of the support of $H^2$,
and we are done.

We include a formal proof here, as suggested by a referee.

\begin{proof}
We follow the notations of \cite[Theorem 6.3.2]{EG23} closely.
The Herr complex $C^\bullet$ (supported in degrees $[0,2]$) computes $H^\bullet(G_K, F(r^{\univ}))$.
Since $B^2$ equals to $C^2$ over the generic fiber $U=X[1/p]$,
by \cite[Tag 0815]{stacks-project},
there exists a $U$-admissible blowup $\pi:\wt{X}\to X$ such that $\pi^*B^2$ is locally free.
Let $\wt{C}^\bullet$ be the pullback complex $\pi^*(C^\bullet)$.
The corresponding $2$-coboundaries $\wt B^2=\pi^*B^2$ (since it is the highest degree coboundary).
Thus the $1$-cocycles $\wt Z^1$ is locally free and $[\wt C^0\to \wt Z^1]$ is a good complex.

Lifting the class $[\bar c]$ to an element of $\kappa\otimes C^1$ (where $\kappa$ is the residue field of $R$)
and then to an element $c$ of $C^1$.
$c$ can be thought of as a homomorphism $R\to C^1$ whose image under the coboundary lies in $m_RC^2$.
The composite $b:R\xrightarrow{c}C^1\to B^2$ pulls back to a section
$\wt{b}:\cO_{\wt X}\to \wt B^2$.
By \cite[Lemma 6.2.7]{EG23} and the SGR property, $\wt b$ has non-empty zero locus,
which contain a point $\wt x$ lying over the closed point $x\in X$.
The section $c$ pulls back to a section $\wt c:\cO_{\wt X}\to \wt C^1$,
whose valued at the point $\wt x$ lies in the fiber of $\wt Z^1$.
In other words, the fiber of $\wt c$ at $\wt x$ defines a 1-cocycle
in the complex $\kappa(\wt x)\otimes [\wt C^0\to \wt Z^1]$,
giving rise to a class $\bar e\in H^1(\kappa(\wt x)\otimes [\wt C^0\to \wt Z^1])$
lifting the original class $[\bar c]$.

Since $\wt X$ is $\bZ_p$-flat,
there exists a morphism $\wt f: \spec \bZp\to \wt X$ lifting $\wt x$.
The composite $f:\spec \bZp\xrightarrow{\wt f}\wt X\to X$
lifts the closed point $x\in X$,
and determines an $L$-parameter $r^\circ:G_K\to L(\bZp)$.
Since $H^2(\wt C^\bullet)$ is the kernel of the homomorphism of locally free sheaves $\wt B^2\hookrightarrow \wt C^2$
and is torsion,
by \cite[Lemma 6.2.1]{EG23} there is an effective Cartier divisor $D$
contained in the special fiber of $\wt X$ with the property that
for any morphism to $\wt X$ that meets $D$ properly, the higher derived pullbacks
of $H^2(\wt C^\bullet)$ under this morphism vanish.
Since $\wt f$ meets the special fiber of $\wt X$ properly and thus meets $D$ properly,
$\mathbf{L}_i\wt f^*H^2(\wt C^\bullet)=0$ for $i>0$.
Thus
$$
H^1(G_K, F(r^\circ))=H^1(\wt f^*\wt C^\bullet)=\wt f^*H^1(\wt C^\bullet)
=H^1(\wt f^*[\wt C^0\to \wt Z^1]).
$$
(See the last two paragraphs of the proof \cite[Theorem 6.3.2]{EG23} for explanations).
Choose a class $e\in H^1(\wt f^*[\wt C^0\to \wt Z^1])$ lifting $\bar e$,
which corresponds to a $1$-cocycle $c$ lifting $\bar e$ by the identifications above.
\end{proof}

\subsection{A non-abelian lifting theorem}
~

\subsubsection{Theorem}
\label{thm:heisenberg-lift}
Let $U$ be a unipotent linear algebraic group of class $2$
whose center is isomorphic to $\Ga$.
Write $Z(U)$ for the center of $U$ and $U^{\ad}$
for $U/Z(U)$.
Fix an algebraic group homomorphism
$\phi:L \to \Aut(U)$ with graded pieces
$\phi^{\ad}: L\to \GL(U^{\ad})$
and $\phi^z: L \to \GL(Z(U))$.

Fix a mod $\varpi$ representation $\bar r:G_K \to L(\bF)$.
Let $[\bar c]\in H^1(G_K, U(\bF))$
be a characteristic $p$ cocycle.

Let $\spec R$ be an irreducible component of a crystalline lifting ring of $\bar r$.

Assume
\begin{itemize}
\item[ {[1]} ]
$H^2(G_K, \phi^{\ad}(r^{\univ}))$ is SGR and is supported on the special fiber of $\spec R$;
\item[ {[2]} ] $p\ne 2$;
\item[ {[3]} ]
There exists a 
finite Galois extension $K'/K$ of prime-to-$p$ degree
such that
$\phi(\bar r)|_{G_{K'}}$
is Lyndon-Demu\v{s}kin; and
\item[ {[4]} ]
There exists a $\bZp$-point
of $\spec R$ which is mildly regular
when restricted to $G_{K'}$.
(In particular, $\spec R[1/p]\ne 0$.)

\end{itemize}
Then there exists a $\bZp$-point of
$\spec R$ which gives rise to
a Galois representation
$r^{\circ}:G_K\to L(\bZp)$ such that
if we endow $U(\bZp)$ with the
$G_K$-action
$G_K\xrightarrow{r^{\circ}} L(\bZp)\xrightarrow{\phi} \Aut(U)(\bZp)$,
the cocycle $[\bar c]$ has a characteristic $0$
lift $[c]\in H^1(G_K, U(\bZp))$.

\begin{proof}
Take $F=\phi^{\ad}$ in
Theorem \ref{thm:abelian-lift}.
The theorem follows from Corollary \ref{cor:obstruction-theory}.
\end{proof}

We explain how the above theorem will be used.
Let $G$ be a connected reductive group over
$\cO_E$.
Let $\bar\rho: G_K\to G(\bF)$
be a mod $\varpi$ representation.
Assume $\bar\rho$ factors through
a parabolic $P\subset G$, with
Levi decomposition $P = L\ltimes U$.
Denote by
$\phi:L\to \Aut(U)$ the conjugation action.
We assume $U$ is unipotent of
class $2$,
so $U^{\ad}$ is an abelian group.
Write 
$\bar r$ for the Levi factor of $\bar\rho$.
$$
\xymatrix{
    & P(\bFp) \ar[d] \\
     G_K \ar[ur]^{\bar\rho}
    \ar[r]^{\bar r} & L(\bFp)
}
$$
Then $\bar\rho$ defines a cohomology class
$[\bar c]\in H^1(G_K, \phi(\bar r))$,
and the theorem above can be used to lift
$[\bar c]$.

\vspace{3mm}

\subsection{An unobstructed lifting theorem}
~

The following result will be used
in the proof of the main theorem.

\subsubsection{Proposition}
\label{prop:unob-lift}
Let $V$ be a unipotent linear algebraic group such that
$V(\bZp)$ is equipped with a continuous $G_K$-action.
Let $[\bar c]\in H^1(G_K, V(\bFp))$
be a characteristic $p$ cocycle.
Let $Z(V)$ be the center of $V$,
and write $V^{\ad}$ for $V/Z(V)$.
The quotient $V\to V^{\ad}$ induces
a map $\ad:H^1(G_K, V)\to H^1(G_K, V^{\ad})$.
Assume 
$H^2(G_K, Z(V)(\bFp))=0$.

If $\ad([\bar c])$ admits a lift
in $H^1(G_K, V^{\ad}(\bZp))$,
then $[\bar c]$ admits a lift in
$H^1(G_K, V(\bZp))$.

\begin{proof}
By \cite[Proposition 43]{Se02},
since $Z(V)$ is a central normal subgroup
of $V$, there exists a long exact sequence
of pointed sets
$$
\xymatrix{
H^1(G_K, V(\bZp))
\ar[d]\ar[r]^{\ad} &
H^1(G_K, V^{\ad}(\bZp))
\ar[r]^{\delta}\ar[d] &
H^2(G_K, Z(V)(\bZp)) \ar[d] \\
H^1(G_K, V(\bFp)
\ar[r]^{\ad} &
H^1(G_K, V^{\ad}(\bFp))\ar[r] &
H^2(G_K, Z(V)(\bFp))
}
$$
By Nakayama's Lemma, we have
$H^2(G_K, Z(V)(\bZp))=0$.
In particular, there exists
$[c']\in H^1(G_K, V(\bZp)$
such that $\ad([\bar c]) = \ad([c'])$
mod $\varpi$.
Write $[\bar {c}']$ for $[c']$ mod $\varpi$.
Say $[\bar c] = [\bar {c}'] + [\bar f]$
for some $[\bar f]\in H^1(G_K, Z(V)(\bFp))$
    (recall that $H^1(G_K, V)$ is a $H^1(G_K, Z(V))$-torsor).
Since $H^1(G_K, Z(V)(\bZp))=0$,
there exists a lift $[f]$ of $\bar f$.
The cocycle $[c]:= [c'] + [f]$
is a lift of $[\bar c]$. 
\end{proof}

\vspace{3mm}
\section{Codimension estimates of loci cut out by $H^2$}
\label{sec:codim}

Assume $p>3$.
Let $K/\Qp$ be a finite extension.
Let $E/\Qp$ be a finite extension with
ring of integers $\cO_E$, uniformizer $\varpi$,
and residue field $\bF$.

\vspace{3mm}

\subsection{The Emerton-Gee stack}~
\vspace{3mm}

We follow the notation of \cite{EG23}.
For each $d>0$, \cite{EG23}
constructed the moduli stack
$\cX_d = \mathcal{X}_{K,d}$
of projective \'etale $(\phi, \Gamma_K)$-modules of rank $d$,
which is a finite-type algebraic stack over $\bF$.

We prove a mild generalization of
\cite[Proposition 5.4.4(1)]{EG23}.

Let $T$ be a reduced finite type
$\bFp$-scheme.
Let $f:T\to (\cX_{a, \red})_{\bFp}\times
(\cX_{d, \red})_{\bFp}$
be a morphism.
There is a morphism
$$
\eta:(\cX_{a, \red})_{\bFp}\times
(\cX_{d, \red})_{\bFp}
\to (\cX_{ad, \red})_{\bFp}
$$
sending a pair of $(\phi, \Gamma)$-modules
$M, N$
to their hom module $\Hom_{\phi, \Gamma}(M,N)$,
by the moduli interpretation.
The morphism $\eta(f)$
corresponds to a family $\bar\rho_T$
of rank-$ad$ Galois representations over
$T$.
We assume $H^2(G_K, \bar\rho_{\eta(t)})$
is of constant rank for all $t\in T(\bFp)$.
By \cite[Lemma 5.4.1]{EG23},
the coherent sheaf
$H^2(G_K, \bar\rho_T)$
is locally free of rank $r$ as
an $\cO_E$-module.

By \cite[Theorem 5.1.22]{EG23},
we can choose a complex of finite rank 
locally free $\cO_E$-modules
$$
C_T^0\to C_T^1\to C_T^2
$$
computing $H^{\bullet}(G_K, \bar\rho_T)$.
Since $H^2(G_K, \bar\rho_T)$
is a locally free sheaf,
the truncated complex
$$
C_T^0\to Z^1_T
$$
is again a complex of locally free $\cO_T$-modules.
The vector bundle 
$\mathscr{V}(Z^1_T):=
\underline{Spec}(\Sym(Z^1_T)^{\vee})$
associated to the locally free sheaf
$Z^1_{T}$ parameterizes all extensions
$$
0\to \bar\rho_{\eta(t)} \to ? \to \bFp \to 0,
\hspace{3mm}t\in T(\bFp)
$$
of the trivial $G_K$-representation $\bFp$
by $\bar\rho_{\eta(t)}$.
There are two projection morphisms
$$()_1:
(\cX_{a, \red})_{\bFp}\times
(\cX_{d, \red})_{\bFp}
\to (\cX_{a,\red})_{\bFp}$$
and
$$()_2:
(\cX_{a, \red})_{\bFp}\times
(\cX_{d, \red})_{\bFp}
\to (\cX_{d,\red})_{\bFp}$$
For each $t\in T(\bFp)$,
$f(t)_1 \in (\cX_{a, \red})({\bFp})$
corresponds to a rank-$a$ Galois representation $\bar\rho_{t_1}$,
and
$f(t)_2 \in (\cX_{d, \red})({\bFp})$
corresponds to a rank-$d$ Galois representation $\bar\rho_{t_2}$.
We have 
$\bar\rho_{\eta(t)}=\Hom_{G_K}(\bar\rho_{t_1}, \bar\rho_{t_2})$.
So we can also regard $\mathscr{V}(Z^1_T)$
is a scheme parametrizing all extensions
$$
0\to \bar\rho_{t_1}\to ? \to \bar\rho_{t_2}\to 0,
\hspace{3mm}t\in T(\bFp)
$$
and we have a morphism
sending extension classes
to equivalence classes of $G_K$-representations
$$
g:\mathscr{V}(Z^1_T)\to (\cX_{a+d, \red})_{\bFp}.
$$

\paragraph{Lemma}
\label{lem:est}
Let $e$ denote the dimension of
the scheme-theoretic image of $T$
in 
$
(\cX_{a, \red})_{\bFp}\times
(\cX_{d, \red})_{\bFp}
$.
Then the scheme-theoretic image of
$V=\mathscr{V}(Z^1_T)$
in $(\cX_{a+d, \red})_{\bFp}$
has dimension at most
$$
e+r+ad[K:\Qp].
$$

\begin{proof}
Without loss of generality,
we assume $T$ (and hence $V$) is irreducible.
The proof is a routine calculation
using stacks. We follow the proof
of \cite[Proposition 5.4.4]{EG23} closely.

Let $v\in V(\bFp)$.
Write $t$ for the composite
$\spec \bFp \xrightarrow{v} V\to T$.
Write $f(t)$ for the composite $f\circ t$
Write $g(v)$ for the composite $g\circ v$.
Define
$$T_{f(t)}:= T\underset{f,
(\cX_{a, \red})_{\bFp}\times
(\cX_{d, \red})_{\bFp}
,f(t)}{\times}\spec \bFp$$
$$V_{g(v)}:= V\underset{g,
(\cX_{a, \red})_{\bFp}\times
(\cX_{d, \red})_{\bFp}
,g(v)}{\times}\spec \bFp$$
$$V_{f(t),g(v)}:= V_{g(v)}\underset{
(\cX_{a, \red})_{\bFp}\times
(\cX_{d, \red})_{\bFp}
,f(t)}{\times}\spec \bFp.$$
Note that
$V_{f(t), g(v)}\cong T_{f(t)}\times_TV_{g(v)}$.

By \cite[Tag 0DS4]{stacks-project},
it suffices to show, for $v$ lying
in some dense open subset of $V$,
$$
\dim V_{f(t),g(v)}
\ge \dim V -(e+r+ad[K:\Qp]).
$$
Let $\bar\rho_{f(t)_1}$ denote
the Galois representation 
corresponding to 
$f(t)_1:\spec \bFp\to (\cX_{a,\red})_{\bFp}$.
Let $\bar\rho_{f(t)_2}$ denote
the Galois representation 
corresponding to 
$f(t)_2:\spec \bFp\to (\cX_{d,\red})_{\bFp}$.
Say $G_{t_1}:=\Aut(\bar\rho_{f(t)_1})$,
and $G_{t_2}:=\Aut(\bar\rho_{f(t)_2})$.
The morphism $f(t)$ factors through
a monomorphism
$$
[\spec \bFp/G_{t_1}]
\times [\spec \bFp/G_{t_2}]
\hookrightarrow
(\cX_{a,\red})_{\bFp}
\times(\cX_{d,\red})_{\bFp}
$$
which induces a monomorphism
$$
([\spec \bFp/G_{t_1}]
\times [\spec \bFp/G_{t_2}])
\underset{(\cX_{a,\red})_{\bFp}\times (\cX_{d,\red})_{\bFp}}{\times}V_{g(v)}
\hookrightarrow
V_{g(v)}.
$$
So it suffices to show
$$
\text{($\dagger$)}\hspace{5mm}\dim V_{f(t),g(v)} \ge \dim V
-(e+r+ad[K:\Qp]) + \dim G_{t_1}+\dim G_{t_2}
$$
for $v$ lying in a dense open of $V$.

There exists an \'etale cover
$S$ of $(T_{f(t)})_{\red}$
such that the pull-back family
$\bar\rho_S$ is a trivial family
with fiber $\bar\rho_t$.

Let $C^0_S\to Z^1_S$ denote the
pullback family of $C^0_T\to Z^1_T$
to $S$.
$C^0_S\to Z^1_S$ is also the
pullback family of the fiber
$C^0_{t}\to Z^1_{t}$ to $S$.
Write $W$ for the affine scheme associated to
$H^1(G_K, \bar\rho_{f(t)_1}^{\vee}\otimes\bar\rho_{f(t)_2})$.
By the isomorphism
$$
H^1(G_K, \bar\rho_{f(t)_1}^{\vee}\otimes\bar\rho_{f(t)_2})\cong \Ext_{G_K}(\bar\rho_{f(t)_1},\bar\rho_{f(t)_2})
$$
there is a morphism $W\to(\cX_{a+d,\red})_{\bFp}$.
Denote by $w$ the image of $v$ in $w$.
We have
$$
S\times_T V_{g(v)} = S\times_TV\times_WW_{h(w)}.
$$
Let $V'$ be the kernel of
$S\times_TV\to S\times_{\bFp}W$,
which is a trivial vector bundle over $S$.
We have
\begin{eqnarray}
\dim V_{f(t),g(v)}&=&\dim S\times_TV_{g(v)}\nonumber\\
&=&\rank V' + \dim S + \dim W_{h(w)}\nonumber\\
&=&\rank Z^1_T-\dim H^1(G_K,\bar\rho_{f(t)_1}^{\vee}\otimes\bar\rho_{f(t)_2})+\dim S+\dim W_{h(w)}\nonumber
\end{eqnarray}
Note that $\dim V-\dim T=\rank Z^1_T$,
and by local Euler characteristic
$
H^0(G_K, \bar\rho_{f(t)_1}^{\vee}\otimes\bar \rho_{f(t)_2})-
H^1(G_K, \bar\rho_{f(t)_1}^{\vee}\otimes\bar \rho_{f(t)_2})
+ r =- ad[K:\Qp]
$.
We can replace $T$ by a dense open of $T$
where $e=\dim T-\dim T_{f(t)}=\dim T-\dim S$.
Combine all these equalities, ($\dagger$)
becomes
$$
\dim W_{h(w)} \ge \dim H^0(G_K, \bar\rho_{f(t)_1}^{\vee}\otimes\bar\rho_{f(t)_2})+\dim G_{t_1} + \dim G_{t_2}
$$
which follows from the fact that
$$
H^0(G_K, \bar\rho_{f(t)_1}^{\vee}\otimes\bar\rho_{f(t)_2}) \rtimes (G_{t_1}\times G_{t_2}) \subset \Aut(\bar\rho_w)
$$
and $\dim W_{h(w)}\ge \dim \Aut(\bar\rho_w)$.
\end{proof}

We recall some terminology from
\cite{EG23}.
Denote by $\ur_x:\Gm\to\cX_1$
the family of unramified characters of
$G_K$.
Let $T$ be a reduced finite type
$\bF$-scheme.
Let $T\to\cX$ be a morphism,
corresponding to a family
$\bar \rho_T$ of $G_K$-representations
over $T$.
We can construct the family of
unramified twisting
$
\bar \rho_T \boxtimes \ur_x
$
over $T\times \Gm$.
$\bar\rho_T$ is said to be \textit{twistable}
if whenever $\bar\rho_t\cong \bar\rho_{t'}\otimes \ur_a$ for
    $t,t'\in T(\bFp)$ and
    $a\in \bFp^{\times}$, we have $a=1$.
 $\bar\rho_T$ is said to be \textit{essentially twistable} if for each $t\in T(\bFp)$, the set of $a\ne 1$ for which $\bar\rho_t\cong \bar\rho_{t'}\otimes \ur_a$
 is finite.

 We say $\bar\rho_T$ is \textit{untwistable} if $\bar\rho$ is not essentially twistable.

From now on, 
write $\cX = (\cX_{2, \red})_{\bFp}$
for the moduli stack
parameterizing $(\phi, \Gamma)$-modules
of rank $2$.

Let $\bar r^{\univ}$ be the universal family
of $(\phi, \Gamma)$-modules
over $\cX$.

\subsubsection{Remarks on the word use ``locus''}
Let (P) be a property
that can be written as
$$
(P) = (P1) - (P2)
$$
where both (P1) and (P2) are closed conditions.

If $\cX$ be a moduli stack of finite type over $\bFp$,
the the {\it locus of objects satisfying property (P$i$)}
is by definition the scheme-theoretic of a finite type morphism
$Y\to \cX$ such that
all objects of $\cX(\bFp)$ satisfying property (P$i$)
are in the image of $Y(\bFp)$,
$i=1,2$.

The {\it locus of objects satisfying property (P)}
is by definition the 
locus of objects satisfying (P1) $-$
locus of objects satisfying (P2).

\subsection{Loci cut out by $H^2(G_K, \sym^3/\det^2)$}~

\vspace{3mm}

Write $H^2$ for
$
H^2(G_K, \frac{\sym^3(\bar {r}^{\univ})}{\det(\bar{r}^{\univ})^2})
$.
Let $x\in \cX(\bFp)$
with corresponding Galois representation
$\bar r_x:G_K\to \GL_2(\bFp)$.

We are interested in $H^2(G_K, \sym^3/\det^2)$ because it is a composition factor of
the unipotent radical of the short root parabolic of the exceptional group $G_2$,
regarded as a representation of the corresponding Levi factor.

\vspace{3mm}

\paragraph{Lemma}
\label{lem:h2-irr}
If
$\bar r_x$ is irreducible, then
$$
h^2_x:=\dim_{\bFp} H^2(G_K, \frac{\sym^3(\bar {r}_x)}{\det(\bar{r}_x)^2})\le 2.
$$

\begin{proof}
An irreducible mod $\varpi$ representation
is of the shape
$\Ind_{G_{K_2}}^{G_K}\bar\chi$
for some character $\bar\chi$ of
    the degree-$2$ unramified extension
    $K_2$ of $K$.
    A direct computation shows
    $$
    \sym^3(\bar{r}_x)=
    \Ind(\bar{\chi}^3) \oplus \Ind(\bar\chi\det\bar {r}_x).
    $$
Both
$H^2(G_K, \frac{\Ind(\bar\chi^3)}{\det(\bar{r}_x)^2})$
and
$H^2(G_K, \frac{\Ind(\bar\chi\det \bar{r}_x)}{\det(\bar{r}_x)^2})$
has dimension at most $1$.
This is because the induction
of a character can't  be
a direct sum of two isomorphic characters
(when $p\ne 2$),
by Shapiro's lemma and local Tate duality.
\end{proof}

\paragraph{Corollary}
$H^2$ is SGR when restricted to the
irreducible locus.

\begin{proof}
Up to unramified twist,
there are only finitely many irreducible
representations.
By Lemma \ref{lem:h2-irr}, we have
$h^2_x\le 2$ when $\bar r_x$ is irreducible.

We first consider the sublocus where
$h^2_x = 2$.
This sublocus consists of finitely many
irreducible $G_K$-representations.
Thus the sublocus in question is the scheme-theoretic union of
the scheme-theoretic images of finitely many morphisms
$\spec \bFp\to \cX$ corresponding to the finitely many irreducibles.
The automorphism group of such an irreducible representation is $\Gm$
and the morphisms $\spec \bFp\to \cX$ factor through $[\spec \bFp/\Gm]\to \cX$.
The sublocus has dimension at most $-1$.

Then we consider the locus where
$h^2_x \le 1$.
This sublocus consists of 
the unramified twists of finitely many
irreducible $G_K$-representations.
Thus the sublocus in question is the scheme-theoretic union of
the scheme-theoretic images of finitely many morphisms
$\Gm\to [\Gm/\Gm] \to \cX$ corresponding to the finitely many irreducibles,
and has dimension at most $\dim~[\Gm/\Gm]=0$.

In either case, $\dim$ of locus $\le [K:\Qp]-h^2_x$.
\end{proof}

\paragraph{Lemma}
\label{lem:h2-ext}
If $\bar {r}_x$ is  a non-trivial
extension of two characters, then
$$
h^2_x:=\dim H^2(G_K, \frac{\sym^3(\bar {r}_x)}{\det(\bar{r}_x)^2})\le1
$$
and when the equality holds,
the quotient character of $\bar {r}_x$
is a character whose third power
is $\bFp(1)$.

\begin{proof}
This is where we make use of
the assumption $p > 3$.
Say $\bar{r}_x \sim
\begin{bmatrix}
\bar{\chi}_1 & \bar c \\
& \bar{\chi}_2
\end{bmatrix}
$.
We claim
$$
\sym^3(\bar{r}_x) \sim
\begin{bmatrix}
\bar{\chi}_1^3 & \bar{\chi}_1^2 \bar c & * & *\\
& \bar{\chi}_1^2\bar{\chi}_2&2\bar{\chi}_1\bar{\chi}_2\bar c & * \\
& & \bar{\chi}_1\bar\chi_2^2 & 3\bar{\chi}_2^2\bar c\\
& & & \bar{\chi}_2^3
\end{bmatrix}
$$
which has a unique $G_K$-invariant quotient line.
Let $\{e_1, e_2\}$ be a basis
of the representation space of $\bar {r}_x$ such that $e_1$ is an invariant line.
Then $\{e_1^3, e_1^2e_2,e_1e_2^2,e_2^3\}$
is a basis of the representation space
of $\sym^3(\bar{r}_x)$.
By duality, we only need to show
$\sym^3(\bar {r}_x)$ has a unique
invariant line.
Clearly $\{e_1^3\}$ defines an invariant line.
Assume there is another invariant line
$\Span(v)$.
We quotient $\sym^{3}(\bar {r}_x)$
by $\Span (e_1^3)$.
The quotient representation has a unique
invariant line generated by the
image of $e_1^2e_2$ (we postpone the explanation to the next paragraph).
So $v\in \Span(e_1^3, e_1^2e_2)$.
But then we must have $v\in \Span(e_1^3)$,
since $[\bar c]$ is a non-trivial
extension class.

The quotient representation $\sym^{3}(\bar {r}_x)/\Span (e_1^3)$ has a $G_K$-invariant line spanned by the image of $e_1^2e_2$.
Say $\Span (u)$ is another invariant line of $\sym^{3}(\bar {r}_x)/\Span (e_1^3)$.
We have $u\in \Span(e_1^2e_2, e_1e_2^2)\cong \bar\chi_1\bar\chi_2\otimes \bar r_x$.
Thus $u\in \Span(e_1^2e_2)$ since $\bar c$ is a non-trivial extension.
\end{proof}

\vspace{3mm}

\paragraph{Corollary}
$H^2$ is SGR when restricted to the
locus where
$\bar {r}_x$ is a non-trivial extension
of two characters.

\begin{proof}
Say $\bar r_x$ is the extension of
$\bar \beta$ by $\bar \alpha$.
By Lemma \ref{lem:h2-ext}, we have
$h^2_x\le 1$ when $\bar r_x$ is a nontrivial extension of characters.
So the locus where $\bar{r}_x$
is a non-trivial extension of characters
consists of four sub-loci:
\begin{itemize}
\item[(i)]
$h^2_x = 1$ and $\Ext^2(\beta, \alpha)=0$;
\item[(ii)]
$h^2_x = 1$ and $\Ext^2(\beta, \alpha)\ne0$;
\item[(iii)]
$h^2_x = 0$ and $\Ext^2(\beta, \alpha)=0$; and
\item[(iv)]
$h^2_x = 0$ and $\Ext^2(\beta, \alpha)\ne 0$;
\end{itemize}

Let $T\subset (\cX_{1,\red})_{\bFp}
\times (\cX_{1,\red})_{\bFp}$
be the locus of the pair $(\alpha, \beta)$,
$\alpha, \beta\in \cX_{1,\red}(\bFp)$;
say $\dim T = e$, and
$\dim \Ext^2(\beta, \alpha) = r$.
By Lemma \ref{lem:est}, each sub-locus
has dimension at most
$$
e+r+[K:\Qp].
$$
In sub-locus (i),
$\beta$ has only finitely many choices once $\alpha$ is chosen,
so $e=-1$, $r=0$;
in sub-locus (ii),
both $\beta$ and $\alpha$ have only
finitely many choices, so
$e=-2$, $r=1$;
in sub-locus (iii),
both $\beta$ and $\alpha$ can vary in a dense open of $(\cX_{1,\red})_{\bFp}$,
so $e=2\dim (\cX_{1,\red})_{\bFp}=0$,
$r=0$;
in sub-locus (iv),
when $\alpha$ is chosen, $\beta$ has only finitely many choices, so
$e=-1$, $r=1$.
We can verify that in each case
$e+r+[K:\Qp] \le \dim \cX -h^2_x=
[K:\Qp]-h^2_x$.
\end{proof}

\paragraph{Lemma}
\label{lem:h2-dist}
If $\bar{r}_x$ is a direct sum of distinct characters,
then
$$
H^2(G_K, \frac{\sym^3(\bar {r}_x)}{\det(\bar{r}_x)^2})\le 2.
$$

\begin{proof}
Say $\bar{r}_x \sim
\begin{bmatrix}
\bar{\chi}_1 &  \\
& \bar{\chi}_2
\end{bmatrix}$.
We have
$$
 \frac{\sym^3(\bar {r}_x)}{\det(\bar{r})^2}
 \cong
 \bar{\chi}_1\bar{\chi}_2^{-2}\oplus
 \bar{\chi}_2^{-1}\oplus \bar{\chi}_1^{-1}
 \oplus \bar{\chi}_2\bar{\chi}_1^{-2}.
$$
If $\bar\chi_1\ne\bar\chi_2$,
then the multiset
$\{
 \bar{\chi}_1\bar{\chi}_2^{-2},
 \bar{\chi}_2^{-1}, \bar{\chi}_1^{-1},
 \bar{\chi}_1^{-2}\bar{\chi}_2\}
$
contains at most
$2$ isomorphic characters.
\end{proof}

\paragraph{Corollary}
$H^2$ is SGR when restricted to the locus
where $\bar {r}_x$ is a direct sum of
distinct characters.

\begin{proof}
By Lemma \ref{lem:h2-dist},
we have $h^2_x \le 2$ when 
$\bar x = \alpha\oplus \beta$
is a direct sum of distinct characters.

In the sublocus where $h^2_x = 2$,
we must have $\pm \alpha=\pm \beta=\bF(-1)$.
The sublocus is the scheme-theoretic union of the scheme-theoretic image
of finitely many $\spec \bFp \times\spec\bFp\to \cX$
and has dimension $0-2=-2$.

In the locus where $h^2_x =1$, we have 
one of the following:
(i) $\alpha=\bF(-1)$, (ii)
$\beta=\bF(-1)$, (iii)
$\alpha=\beta^2(-1)$, (iv)
$\beta=\alpha^2(-1)$.
In each of these cases,
the locus has dimension
$\dim \Gm - \dim \text{Aut}(\bar {r}_x)
= 1-2=-1$.

In the locus where $h^2_x=0$,
both $\alpha$ and $\beta$ lives 
in an untwistable family, and the locus
has dimension
$2\dim \Gm - \dim \text{Aut}(\bar {r}_x)
=2-2=0$.
\end{proof}

\paragraph{Lemma}
\label{lem:h2-scalar}
If $\bar {r}_x$ is a direct sum of
isomorphic characters, then
$$
H^2(G_K, \frac{\sym^3(\bar {r}_x)}{\det(\bar{r}_x)^2})\le 4.
$$

\begin{proof}
This is trivial because the underlying
$\bFp$-vector space is $4$-dimensional.
\end{proof}

\paragraph{Corollary}
$H^2$ is SGR when restricted to the locus
where $\bar {r}_x$ is a direct sum of
isomorphic characters.

\begin{proof}
The automorphism group is $4$-dimensional.
So the locus in the moduli stack  has dimension
$\dim \Gm - \dim \text{Aut}(\bar {r}_x)
=1-4=-3$.
\end{proof}

\vspace{3mm}

\paragraph{Theorem}
\label{thm:egstack}
The locus of $\bar r_x$ in $\cX$ for which
$$
H^2(G_K, \frac{\sym^3(\bar {r}_x)}{\det(\bar{r}_x)^2})\ge r
$$
is of dimension
at most
$[K:\Qp]-r$.

\begin{proof}
This theorem follows immediately from
Lemma \ref{lem:h2-irr},
Lemma \ref{lem:h2-ext},
Lemma \ref{lem:h2-dist},
Lemma \ref{lem:h2-scalar},
and their corollaries.
\end{proof}

Fix a mod $\varpi$ representation
$
\bar r: G_K\to GL_2(\bF)
$.
Let $\underline{\lambda}$ be a Hodge type.
Let $R$ be an irreducible component of the crystalline lifting ring
${R_{\bar r}^{\crys, \underline{\lambda}, \cO_E}}$.
Assume $\spec R[1/p]\ne \emptyset$.
Let $r^{\univ}$ be the universal family
of Galois representations on $R$.

Since $H^2(G_K, \frac{\sym^3({r}^{\univ})}{\det({r}^{\univ})^2})$ is a coherent sheaf,
by the semicontinuity theorem,
the locus
$ X_s:=\{x\in\spec R|\dim \kappa(x)\otimes_R H^2 \ge s\}$
is locally closed, and 
has a reduced induced scheme structure.

\vspace{3mm}
\paragraph{Theorem}
\label{thm:codim-ad}
Let $R$ be an irreducible component of the crystalline lifting ring
of regular labeled Hodge-Tate weights.
If $H^2(G_K, \frac{\sym^3({r}^{\univ})}{\det({r}^{\univ})^2})$ is $\varpi$-torsion,
the locus $$\{x\in\spec R|\dim \kappa(x)\otimes_R H^2(G_K, \frac{\sym^3({r}^{\univ})}{\det({r}^{\univ})^2}) \ge s\}$$ has codimension
$\ge s+1$ in $\spec R$ for $s\ge 1$.

\begin{proof}
The proof is identical to that
of \cite[Theorem 6.1.1]{EG23}
if we use 
Theorem \ref{thm:egstack} instead of
\cite[Theorem 5.5.12]{EG23}.
\end{proof}

\vspace{3mm}

\vspace{3mm}

\section{The existence of crystalline lifts for 
the exceptional group $G_2$}
\label{sec:main}

\subsection{Parabolics of $G_2$}~
\label{ss:G2}

Let $G_2$ be the 
Chevalley group over $\cO_E$
of type $G_2$.

Let $E/\Qp$ be a finite extension
with ring of integers $\cO_E$,
residue field $\bF$ and uniformizer
$\varpi$.

We remind the reader of the root
system of $G_2$:
\begin{center}
\begin{tikzpicture}
    \foreach\ang in {60,120,...,360}{
     \draw[->,blue!80!black,thick] (0,0) -- (\ang:2cm);
    }
    \foreach\ang in {30,90,...,330}{
     \draw[->,blue!80!black,thick] (0,0) -- (\ang:3cm);
    }
    \draw[magenta,->](1,0) arc(0:150:1cm)node[pos=0.1,right,scale=0.5]{$5\pi/6$};
    \node[anchor=south west,scale=0.6] at (2,0) {$\alpha$};
    \node[anchor=south west,scale=0.6] at (-2.9,1.6) {$\beta$};
    \node[anchor=south west,scale=0.6] at (-.9,1.6) {$\beta+\alpha$};
    \node[anchor=south west,scale=0.6] at (.9,1.6) {$\beta+2\alpha$};
    \node[anchor=south west,scale=0.6] at (2.5,1.6) {$\beta+3\alpha$};
    \node[anchor=south west,scale=0.6] at (0.2,2.8) {$2\beta+3\alpha$};
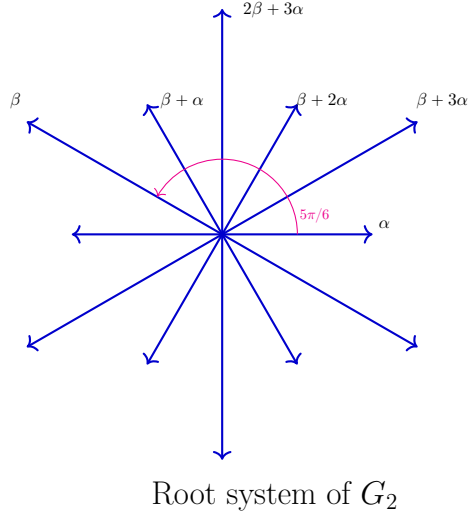
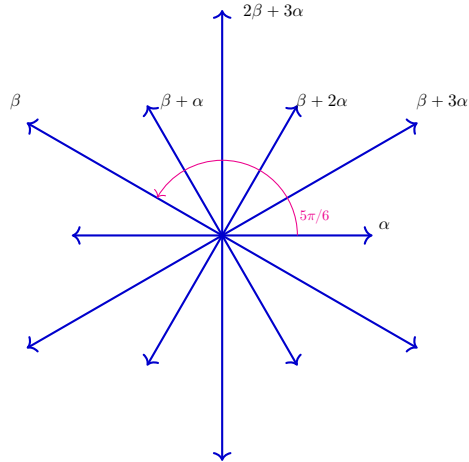
    \node[anchor=north,scale=0.6] at (0.7,-3.2) {\huge Root system of $G_{2}$};
  \end{tikzpicture}
\end{center}

\subsubsection{The short root parabolic}

Let $P\subset G_2$ be the short root parabolic,
which admits a Levi decomposition
$P=L\ltimes U$.
The Levi factor $L$ is a copy of $\GL_2$
and the unipotent radical $U$ is a unipotent group of class $2$.
Write $U^{\ad}$ for $U/Z(U)$.

Fix an isomorphism $\std:L\cong \GL_2$.
We have
\begin{itemize}
\item $Z(U)\cong \mathbb{G}_a$, and
\item $U^{\ad}\cong \mathbb{G}_a^{\oplus 4}$.
\end{itemize}
Write $\Lie U = Z(U)\oplus U^{\ad}$.
The Levi factor acts on $U$ by conjugation.
We have an isomorphism of $L$-modules
$$
(*) \hspace{5mm} \Lie U \cong \frac{1}{{\det}^2}\sym^3(\std) \oplus
\frac{1}{\det}
$$
where $\det:L\to \Gm$ is the determinant
character, and $\std:L\xrightarrow{\cong} \GL_2$
is the fixed isomorphism.
The above short exact sequence 
can be upgraded to a short exact sequence of
groups with $L$-actions
$$
0\to \frac{1}{\det}\to U
\to \frac{1}{{\det}^2}\sym^3(\std) \to 0.
$$

For lack of reference, we explain how to get $(*)$.
By inspecting the root system for $G_2$,
we find that the roots whose root group is contained in $U^{\ad}$
lie in a single line.
Therefore $U^{\ad}$ is an irreducible $L$-module,
and is thus isomorphic to $\sym^3(\std)$
up to an algebraic character;
then computation shows the character is $1/\det^2$
(also see the SageMath code on my homepage).

\subsubsection{The long root parabolic}

Let $Q\subset G_2$ be the long root parabolic,
which admits a Levi decomposition
$Q=L'\ltimes V$ where
$L'\cong \GL_2$ and $V$ is a unipotent group of class $3$.
Fix an isomorphism $\std:L'\xrightarrow{\cong}\GL_2$.
Write $\det$ for the composition
$L'\xrightarrow{\std} \GL_2 \xrightarrow{\det}\GL_1$.

Write $U'$ for $V/Z(V)$.
Then $U'$ is a unipotent group of class $2$ whose
center is isomorphic to $\Ga$.
The conjugation action of $L'$ on $U'$ is given by
$U'/Z(U') \cong \std$,
and $Z(U') \cong \det$,
as $L'$-modules.

\vspace{3mm}

\subsection{Theorem}
\label{thm:existence-crys-lift-G2}
Assume $p > 3$.
Let $K/\Qp$ be a $p$-adic field.
Let $\bar\rho:G_K\to G_2(\bFp)$
be a mod $\varpi$ Galois representation.
Then $\bar\rho$ admits
a crystalline lift
$
\rho^{\circ}:G_K\to G_2(\bZp)
$
of $\bar\rho$.

Moreover, if $\bar\rho$ factors through 
a maximal parabolic and the Levi factor $\bar r:=\bar r_{\bar\rho}$ of $\bar\rho$
admits a Hodge-Tate regular and crystalline lift $r_1$ such that
the adjoint representation $\phi^{\Lie}(r_1)$ has Hodge-Tate weights
slightly less than $\underline 0$,
then $\rho^{\circ}$ can be chosen such that it factors through
the same maximal parabolic and its Levi factor $r_{\rho^{\circ}}$
lies on the same irreducible component of the spectrum of the crystalline
lifting ring that $r_1$ does.

\begin{proof}

If $\bar\rho$ is irreducible, then
$\bar\rho$ admits a crystalline lift
by \cite{L22}.

The exceptional group $G_2$ has two maximal
parabolic subgroups: the short root parabolic,
and the long root parabolic.

If $\bar\rho$ is reducible, then it factors through either parabolic subgroups.

\subsubsection{The short root parabolic case}
\label{case:short-root}
~
\vspace{3mm}

Let $P\subset G_2$ be the short root parabolic.
Recall that $P$ has a Levi decomposition
$P = L \ltimes U$.
Fix an isomorphism
$L\cong \GL_2$.

By Lemma \ref{lem:K-primed},
there exists a finite Galois extension
$K'/K$, of prime-to-p degree
such that $\bar r|_{K'}$ is Lyndon-Demu\v{s}kin.

Write $Z(U)$ for center of $U$,
and write $U^{\ad}$ for $U/Z(U)$.
Write $\phi:L\to\Aut(U)$ for the conjugation action,
with graded pieces
$\phi^{\ad}:L\to \GL(U^{\ad})$
and $\phi^{z}:L\to\GL(Z(U))$.
Write $\phi^{\Lie}$ for $\phi^{\ad}\oplus \phi^{z}$.
\vspace{3mm}

\paragraph{Lemma}
\label{lem:lift-r0}
Assume $p>2$.
There exists a Hodge-Tate regular crystalline lifting
$r^{\circ}: G_K\to L(\bZp)$
of the Levi factor $\bar r$,
such that
the adoint representation
$
\phi^{\Lie}(r^\circ):
G_K\xrightarrow{r^{\circ}}
L(\bZp)
\to \GL(\Lie U(\bZp))
$
has labeled Hodge-Tate weights
slightly less that $\underline{0}$.

\begin{proof}
It is well-known Hodge-Tate regular crystalline lifts of $\bar r$ exists since $L\cong\GL_2$.
We have
$\phi^{\operatorname{Lie}}(r^\circ) = \frac{1}{{\det r^\circ}^2}\sym^3(r^\circ) \oplus
\frac{1}{\det r^\circ}$.
So by replacing $r^\circ$ by a Tate twist,
we can ensure $\phi^{\operatorname{Lie}}(r^\circ)$ has labeled Hodge-Tate weights
slightly less that $\underline{0}$.

\end{proof}

Let $\spec R$ be an irreducible component (with non-empty generic fiber)
of a crystalline lifting
ring $R^{\crys,\underline{\lambda}}_{\bar r}$
of regular labeled Hodge-Tate weights $\underline \lambda$
such that
the labeled Hodge-Tate weights $\phi^{\Lie}(\underline{\lambda})$ are
slightly less 0.
By the lemma above, such a $\spec R$ exists.

Let $r^{\univ}:G_K \to L(R)$
be the universal Galois representation.

The mod $\varpi$ Galois
representation $\bar r$ defines
a Galois action $\phi(\bar r):G_K\to \Aut(U(\bFp))$ on $U(\bFp)$.
By \ref{par:non-ab-h1},
the datum of $\bar \rho:G_K \to G_2(\bFp)$
is encoded in a non-abelian cocycle
$[\bar c]\in H^1(G_K, U(\bFp))$.

The strategy for lifting $\bar\rho$
is as follows.
We choose a suitable $\bZp$-point
$x$ of $\spec R$ which defines
a lift $r_x:G_K\to L(\bZp)$
of $\bar r$, and endow $U(\bZp)$
with the Galois action
$\phi(r_x):G_K\xrightarrow{r_x}L(\bZp)
\to \Aut(U(\bZp))$.
There is a map of pointed set
$
H^1(G_K, U(\bZp)) \to H^1(G_K, U(\bFp))
$.
If the cohomology class $[\bar c]$
admits a lift $[c]\in H^1(G_K, U(\bZp))$,
then $\bar\rho$ admits
a lift $\rho:G_K\to G_2(\bZp)$ whose datum is encoded in $[c]$.
Such a lift $\rho$ is crystalline by the main result of \cite{L21},
since $\phi^{\Lie}(r^\circ)$ has labeled Hodge-Tate weights
slightly less than $\underline{0}$.

By Theorem \ref{thm:heisenberg-lift},
to lift the non-abelian 1-cocycle $[\bar c]$,
it suffices to verify the following:
\begin{itemize}
\item[ {[1]} ]
$H^2(G_K, \sym^3(r^{\univ})/\det^2(r^{\univ}))$ is SGR and supported on the special fiber of $\spec R$;
\item[ {[2]} ] $p\ne 2$;
\item[ {[3]} ]
There exists a 
finite Galois extension $K'/K$ of prime-to-$p$ degree
such that
$\phi(\bar r)|_{G_{K'}}$
is Lyndon-Demu\v{s}kin; and
\item[ {[4]} ]
There exists a $\bZp$-point
of $\spec R$ which is mildly regular
when restricted to $G_{K'}$.
\end{itemize}
\vspace{3mm}

[1] is verified by Theorem \ref{thm:codim-ad}.
Note that since the Hodge type of $\spec R$ is chosen so that
$\sym^3(r_x)/\det(r_x)^2$
has labeled Hodge-Tate weights
slightly less than $\underline{0}$,
$H^2(G_K, \sym^3(r_x)/\det(r_x)^2)$
is torsion for any characteristic $0$
point $x$ of $\spec R$.
[3] follows from Lemma \ref{lem:K-primed},
and [4] follows from Proposition \ref{prop:G2-MR}.

\vspace{3mm}

\subsubsection{The long root parabolic case}
~
\vspace{3mm}

Let $Q\subset G_2$ be the long root parabolic.
$Q$ has a Levi decomposition
$Q=L'\ltimes V$.
Fix an isomorphism $\std:L'\xrightarrow{\cong}\GL_2$.
Write $\det$ for the composition
$L'\xrightarrow{\std} \GL_2 \xrightarrow{\det}\GL_1$.

Let $\{1\}=V_0\subset V_1\subset V_2\subset V_3=V$ be the upper central series of 
$V$.
Then the conjugation action of $L'$
on each graded piece is given by
\begin{itemize}
\item $V_3/V_2 \cong \det\otimes \std$;
\item $V_2/V_1 \cong \det$;
\item $V_1 \cong \std$.
\end{itemize}

Suppose $\bar\rho$ factors through the
long root parabolic $Q$, but not
the short root parabolic $P$.
Then the Levi factor
$$
\bar r: G_K\xrightarrow{\bar\rho}Q(\bFp)
\to L'(\bFp)
$$
is necessarily an irreducible representation.
If we endow each graded piece of $V(\bFp)$
with the Galois action 
$G_K\xrightarrow{\bar r}L(\bZp)\to \GL(V_{i+1}(\bFp)/V_i(\bFp))$, then we have,
by local Tate duality,
$$
H^2(G_K, V_3(\bFp)/V_2(\bFp))
=H^2(G_K, \bar r \otimes \det \bar r)=0
$$
$$
H^2(G_K, V_1(\bFp))
=H^2(G_K, \bar r)=0
$$
So the only cohomological obstruction
occurs in the second graded piece.

The datum of $\bar \rho$
is encoded in a non-abelian cocycle
$[\bar c]\in H^1(G_K, V(\bFp))$.
Just as is done in the short root parabolic case, it suffices to lift the cocycle
$[\bar c]$.
By Proposition \ref{prop:unob-lift},
since the only cohomological obstruction
lies in the second graded piece, it suffices to lift
$\ad([\bar c]) \in H^1(G_K, (V/V_1)(\bFp))$.

Write $U'$ for $V/V_1$.
Recall that $U'$ is a unipotent group of class $2$
with rank-$1$ center, and we can directly
appeal to Theorem \ref{thm:heisenberg-lift}.
We repeat the procedure worked out
in the short root case \ref{case:short-root}.

Let $r^\circ$ be a lift of $\bar r$ such that $r^\circ$
is Hodge-Tate regular and crystalline and the Hodge-Tate weights of $r^\circ$
are strictly less than $\underline{0}$.

Let $\spec R$ be the irreducible component of the crystalline lifting ring
of $\bar r$ containing $r^\circ$.
Write $r^{\univ}:G_K \to \GL_2(R)$ for the universal family.

Write $Z(U')$ for the center of $U'$, and
write $U^{'\ad}$ for $U'/Z(U')$.
Write $\phi^{\ad}$ for the conjugate action $L'\to \Aut(U^{'\ad})$
and write $\phi^{z}$ for the conjugate action $L' \to \Aut(Z(U'))$.

Note that $\phi^{\ad}(r^{\univ}) = r^{\univ}$
and $\phi^{z}(r^{\univ}) = \det r^{\univ}$.

We have the following check list:
\begin{itemize}
\item[ {[1]} ]
$H^2(G_K, \det(r^{\univ}) r^{\univ})$ is SGR;
\item[ {[2]} ] $p\ne 2$;
\item[ {[3]} ]
There exists a 
finite Galois extension $K'/K$ of prime-tp-$p$ degree
such that
$\phi(\bar r)|_{G_{K'}}$
is Lyndon-Demu\v{s}kin; and
\item[ {[4]} ]
There exists a $\bZp$-point
of $\spec R$ which is mildly regular
when restricted to $G_{K'}$.
\end{itemize}
By the assumption $H^2(G_K, \det(r^{\univ}) r^{\univ})=0$.
[3] follows from Lemma \ref{lem:K-primed},
and [4] follows from Proposition \ref{prop:G2-MR}.
\end{proof}

\vspace{3mm}

\begin{appendix}

\section{Non-denegeracy of mod $\varpi$ cup product
for $G_2$}
\label{sec:input-i}

\label{sec:delta}

Let $\bF$ be a finite field of characteristic $p>3$.
Write $G_2$ for the Chevalley group over $\bF$ of type $G_2$.

Let $P$ be the short root parabolic of $G_2$.
Let $P=L\ltimes U$ be the Levi decomposition.
Let  $\bar r:G_K\to L(\bF)$ be a Galois representation which is Lyndon-Demu\v{s}kin.
Since $L\cong \GL_2$, $\bar r$ is the extension of two trivial characters.

Denote by $\phi:L\to \Aut(U)$ the conjugation action.

$G_K$ acts on $U$ via the conjugate
action $G_K \xrightarrow{r^{\circ}} L\xrightarrow{\phi} \Aut(U)$.

\vspace{3mm}
We set up a computational framework
to prove various claims.
Let $\{x_0, \cdots, x_n, x_{n+1}\}$
be the Demu\v{s}kin generators.

Let $\{e_1
,e_2
\}$
be a basis of the representation space of $\bar r$
such that $r^{\circ}$
is upper-triangular with respect to this basis.
Without loss of generality, assume
$e_1=
\begin{bmatrix}
1 \\ 0
\end{bmatrix}
,e_2 = 
\begin{bmatrix}
0 \\ 1
\end{bmatrix}
$.
Say
for $i=0,\cdots,n+1$,
$
\bar r(x_i)
=
\begin{bmatrix}
1 & l_i \\
& 1
\end{bmatrix}
$.

The set $\{e_1^3,e_1^2e_2,e_1e_2^2,e_2^3\}$
is a basis of the representation space $\sym^3(\bar r)$,
which is identified with $U^{\ad}(\bF)$.

The root system of $G_2$ can be found in Subsection \ref{ss:G2}.
In the diagram, $\alpha$ is the short root, and $\beta$ is the short root.
Each root $x$ generates a root group $U_x\subset U$.
The short root parabolic $P$ has $7$ root
groups: the $5$ root groups 
$$\{U_{\beta}, U_{\beta+\alpha}, U_{\beta+2\alpha}, U_{\beta+3\alpha}, U_{2\beta+3\alpha}\}$$ lying above
the $x$-axis generates the unipotent radical
$U$,
the two root groups $\{U_{\alpha}, U_{-\alpha}\}$ lying
on the $x$-axis are the root groups of the Levi factor group $L$.
Say
under the identification
$\std: L\cong \GL_2$, the matrices
$
\begin{bmatrix}
0 & * \\ 0 & 0
\end{bmatrix}$
are identified with the root group $U_{\alpha}$.
Now that we have identifications
\begin{eqnarray}
\Span e_1^3 &\sim& U_{\beta} \nonumber\\
\Span e_1^2 e_2 &\sim& U_{\beta+\alpha}\nonumber\\
\Span e_1 e_2^2 &\sim& U_{\beta+2\alpha}\nonumber\\
\Span e_2^3 &\sim& U_{\beta+3\alpha}\nonumber
\end{eqnarray}

For ease of notation, write
$E_0:= e_1^3$, 
$E_1:= e_1^2e_2$,
$E_2:= e_1e_2^2$,
$E_3:= e_2^3$.
A basis of 
$$C^1_{\LD}(U^{\ad}(\cO_E))\cong
\{\langle x_0, \cdots, x_{n+1} \rangle
\to 
U_{\beta}(\cO_E)
\oplus
U_{\beta+\alpha}(\cO_E)
\oplus
U_{\beta+2\alpha}(\cO_E)
\oplus
U_{\beta+3\alpha}(\cO_E)\}
$$ is 
given by
$$
\mathscr{B}=
\left\{
\begin{matrix}
x_0^*E_0 ,x_1^*E_0 , &\dots, &x_{n+1}^* E_0,\\
x_0^*E_1 ,x_1^*E_1 , &\dots, &x_{n+1}^* E_1,\\
x_0^*E_2 ,x_1^*E_2 , &\dots, &x_{n+1}^* E_2,\\
x_0^*E_3 ,x_1^*E_3 ,& \dots,& x_{n+1}^* E_3
\end{matrix}
\right\}
$$
where $x_i^*E_j$
is the cochain 
$c:\langle x_0, \cdots, x_{n+1}\rangle$
such that $c(x_k)=\delta_{ik}E_j$,
where $\delta_{ik}$ is the Kronecker delta.
For any $c\in C^1_{\LD}(U^{\ad})$,
we can write down the $\mathscr{B}$-coordinates 
$[c]_{\mathscr{B}} := (c_v)_{v\in\mathscr{B}}$ of $c$.

\vspace{3mm}

\subsubsection{Lemma}
\label{lem:cup-C-nondeg}
The cup products on cochains
$$
\cup_{\bF}:
C^1_{\LD}(U^{\ad}(\bF))
\times
C^1_{\LD}(U^{\ad}(\bF))
\to C^2_{\LD}(Z(U)(\bF))
$$
is non-degenerate.

\noindent
\textbf{Ideas} \hspace{3mm}
We compute the cup products 
$v\cup w$
for $v,w\in\mathscr{B}$.
The matrix $[\cup_{\bF}]_{\mathscr{B}}$
is anti-lower-triangular, (that is,
of the shape
$$
\begin{bmatrix}
0 & 0 & 0 &* \\
0 & 0 & * &* \\
0 & * & * & * \\
* & * & * & *
\end{bmatrix}
$$
whose anti-diagonal blocks are
constant invertible matrices
), and thus non-degenerate.

To help the reader better understand what's
going on,
we attached SageMath code in the Appendix \ref{sec:sage}.

\begin{proof}
Recall the relator of the Lyndon-Demu\v{s}kin group is
$$
R=\DemRel.
$$
Since we are working mod $\varpi$,
we have for any $p>5$, any $g\in G_{K'}$,
$\phi(\bar r(g))^p \equiv \Id$ mod $\varpi$ (See Appendix \ref{sec:sage} for the verification).
In particular, the relator $R$ reduces to
$$
(x_0, x_1)\dots(x_{n}, x_{n+1})
$$
when we compute mod $\varpi$.
(When $p=5$, things are still good,
and can be confirmed by running
the SageMath code in the appendix.)

We regard cochains in
$C^1_{\LD}(U^{\ad}(\bF))$
as a $(U^{\ad}(\bF))$-valued function
on the free group with generators
$\{x_0, \dots, x_{n+1}\}$,

Now we let $c$ be the ``universal''
mod $\varpi$
1-cochain.
That is, we let 
$$
\left\{
\begin{matrix}
\lambda_{0,0}, &\lambda_{1,0} , &\dots, &\lambda_{n+1,0},\\
\lambda_{0,1}, &\lambda_{1,1} , &\dots, &\lambda_{n+1,1},\\
\lambda_{0,2}, &\lambda_{1,2} , &\dots, &\lambda_{n+1,2},\\
\lambda_{0,3}, &\lambda_{1,3} , &\dots, &\lambda_{n+1,3}
\end{matrix}
\right\}
$$
be indeterminants, and set
$$c:= \sum \lambda_{i,j}x_i^*E_j\in
C^1_{\LD}(U^{\ad}(\bF))\otimes \bZ[\lambda_{i,j}].$$

The cup product
$$
c\cup c = Q(c) \in C^2_{\LD}(Z(U)(\bF))\otimes \bZ[\lambda_{i,j}] = Z(U)(\bF)\otimes \bZ[\lambda_{i,j}] \cong \bF[\lambda_{i,j}]
$$ will be a quadratic form
in variables $\{\lambda_{i,j}\}$,
and the matrix of this quadratic form
is nothing but the matrix
$[\cup_{\bF}]_{\mathscr{B}}$.
Recall that $c\cup c =Q(c)$
is defined to be
the projection of $\widetilde{c}(R)$ onto
the center of the Lie algebra $\Lie U$,
where
$\widetilde{c}\in C^1_{\LD}(U(\bF))$
is the unique extension of $c$ to
a $U(\bF)$-valued cochain
as is explained in Section \ref{sec:LD}.

Write $[\cup_{\bF}]_{\mathscr{B}}$ as
a block matrix
$$
\renewcommand{\kbldelim}{(}
\renewcommand{\kbrdelim}{)}
[\cup_{\bF}]_{\mathscr{B}}=
\kbordermatrix{
    & \beta &\omit& \beta+\alpha &\omit&\beta+2\alpha &\omit&\beta+3\alpha\\
\beta &M_{11}&\omit\vrule& M_{12} &\omit\vrule& M_{13} &\omit\vrule& M_{14} \\
\cline{2-8}
\beta+\alpha&M_{21} &\omit\vrule& M_{22} &\omit\vrule& M_{23} &\omit\vrule& M_{24} \\
\cline{2-8}
\beta+2\alpha&M_{31} &\omit\vrule& M_{32} &\omit\vrule& M_{33} &\omit\vrule& M_{34} \\
\cline{2-8}
\beta+3\alpha&M_{41} &\omit\vrule& M_{42} &\omit\vrule& M_{43} &\omit\vrule& M_{44}
}
$$
where each $M_{ij}$ is an $(n+2)\times(n+2)$ matrix.
We say the blocks $M_{24}$, $M_{33}$,
$M_{34}$, $M_{42}$, $M_{43}$, $M_{44}$
are \textit{strictly below the anti-diagonal},
and we call $M_{41}$, $M_{32}$, $M_{23}$
and $M_{14}$ the anti-diagonal blocks.
\begin{figure}[h!]
\begin{minipage}{.49\linewidth}
\centering
$$
\renewcommand{\kbldelim}{(}
\renewcommand{\kbrdelim}{)}
\kbordermatrix{
    & \beta &\omit& \beta+\alpha &\omit&\beta+2\alpha &\omit&\beta+3\alpha\\
\beta &&\omit\vrule&  &\omit\vrule&  &\omit\vrule&  \\
\cline{2-8}
\beta+\alpha& &\omit\vrule&  &\omit\vrule&  &\omit\vrule& M_{24} \\
\cline{2-8}
\beta+2\alpha& &\omit\vrule&  &\omit\vrule& M_{33} &\omit\vrule& M_{34} \\
\cline{2-8}
\beta+3\alpha& &\omit\vrule& M_{42} &\omit\vrule& M_{43} &\omit\vrule& M_{44}
}
$$
\caption{Strictly below anti-diagonal}
\end{minipage}
\begin{minipage}{.45\linewidth}
$$
\renewcommand{\kbldelim}{(}
\renewcommand{\kbrdelim}{)}
\kbordermatrix{
    & \beta &\omit& \beta+\alpha &\omit&\beta+2\alpha &\omit&\beta+3\alpha\\
\beta &&\omit\vrule&  &\omit\vrule&  &\omit\vrule& M_{14} \\
\cline{2-8}
\beta+\alpha& &\omit\vrule&  &\omit\vrule& M_{23} &\omit\vrule&  \\
\cline{2-8}
\beta+2\alpha& &\omit\vrule& M_{32} &\omit\vrule&  &\omit\vrule&  \\
\cline{2-8}
\beta+3\alpha&M_{41} &\omit\vrule&  &\omit\vrule&  &\omit\vrule& 
}
$$
\caption{Anti-diagonal blocks}
\end{minipage}
\end{figure}

\textbf{Sublemma}
Let $g=g_1g_2\dots g_s$.
Write $\phi_i$ for $\phi(\bar r(g_1,\dots,g_{i-1}))$.
We have 
$$\widetilde{c}(g) = \sum \phi_i \widetilde{c}(g_i)
+ \frac{1}{2}\sum_{i<j}[\phi_i \widetilde{c}(g_i), \phi_j \widetilde{c}(g_j)].$$

\begin{quotation}
\begin{proof}
An immediate consequence of the
Baker–Campbell–Hausdorff formula.
\end{proof}
\end{quotation}

\vspace{3mm}

Note that
$\phi(\bar {r}((x_i, x_j)) = \Id$,
so
\begin{eqnarray}
\widetilde{c}(R) &=& \widetilde{c}(\DemRel)
\nonumber\\
&=&\sum \widetilde{c}((x_{2k},x_{2k+1}))
+ \frac{1}{2}\sum_{j<k} 
[\widetilde{c}((x_{2j}, x_{2j+1}), \widetilde{c}( (x_{2k}, x_{2k+1})]\nonumber
\end{eqnarray}
We have 
$$
\widetilde{c}( (x_{2k}, x_{2k+1}))
= -\phi(x_{2k}^{-1})(\phi(x_{2k+1})-1)\widetilde{c}(x_{2k})
+ \phi(x_{2k}^{-1}x_{2k+1}^{-1})(\phi(x_{2k})-1)\widetilde{c}(x_{2k+1})
+ Z_{k}
=Y_k+ Z_k
$$
where $Z_k$ is a sum of Lie brackets
(see below), and
lies in the center of the $\Lie U$.
Note that 
$[Y_j, Y_k]$ only contributes to the part
of $[\cup_{\bF}]_{\mathscr{B}}$
which lies strictly below the anti-diagonal,
because 
$(\phi(x_{2k})-1)$ and $(\phi(x_{2k+1})-1)$
moved the appearance of the inderterminant
$\lambda_{i,j}$
from the root group $U_{\beta+j\alpha}$
to the root group $U_{\beta+(j+1)\alpha}$.

So it remains to analyze
$\sum Z_k$.
We have
\begin{eqnarray}
2 Z_k &=& 
[-\phi(x_{2k}^{-1}) \widetilde{c}(x_{2k}),
-\phi(x_{2k}^{-1}x_{2k+1}^{-1})\widetilde{c}(x_{2k+1})] \nonumber
\\&+&
[-\phi(x_{2k}^{-1}) \widetilde{c}(x_{2k}),
+\phi(x_{2k}^{-1}x_{2k+1}^{-1})\widetilde{c}(x_{2k})]\nonumber\\
&+&
[-\phi(x_{2k}^{-1}) \widetilde{c}(x_{2k}),
+\phi(x_{2k}^{-1}x_{2k+1}^{-1}x_{2k})\widetilde{c}(x_{2k+1})] \nonumber\\
&+&
[-\phi(x_{2k}^{-1}x_{2k+1}^{-1}) \widetilde{c}(x_{2k+1}),
+\phi(x_{2k}^{-1}x_{2k+1}^{-1})\widetilde{c}(x_{2k})]\nonumber\\
&+&
[-\phi(x_{2k}^{-1}x_{2k+1}^{-1}) \widetilde{c}(x_{2k+1}),
+\phi(x_{2k}^{-1}x_{2k+1}^{-1}x_{2k})\widetilde{c}(x_{2k+1})] \nonumber\\
&+&
[\phi(x_{2k}^{-1}x_{2k+1}^{-1}) \widetilde{c}(x_{2k}),
+\phi(x_{2k}^{-1}x_{2k+1}^{-1}x_{2k})\widetilde{c}(x_{2k+1})] \nonumber
\end{eqnarray}
Write
\begin{eqnarray}
2Z_k' &:=&
[-\widetilde{c}(x_{2k}), -\widetilde{c}(x_{2k+1})]\nonumber\\
&+&
[-\widetilde{c}(x_{2k}), \widetilde{c}(x_{2k})]\nonumber\\
&+&
[-\widetilde{c}(x_{2k}), \widetilde{c}(x_{2k+1})]\nonumber\\
&+&
[-\widetilde{c}(x_{2k+1}), \widetilde{c}(x_{2k})]\nonumber\\
&+&
[-\widetilde{c}(x_{2k+1}), \widetilde{c}(x_{2k+1})]\nonumber\\
&+&
[\widetilde{c}(x_{2k}), \widetilde{c}(x_{2k+1})]\nonumber
\end{eqnarray}
$Z_k'$ is obtained by replacing
all Galois action in $Z_k$
by the trivial action.
$Z_k -Z_k'$ only contributes
to the part of $[\cup_{\bF}]_{\mathscr{B}}$
with lies strictly below the anti-diagonal
for a similar reason (a ``shifting'' effect).
It is easy to see that 
$$Z_k' = [\widetilde{c}(x_{2k}), \widetilde{c}(x_{2k+1})] = 
\pm\lambda_{2k,0}\lambda_{2k+1,3}
\pm\lambda_{2k+1,0}\lambda_{2k,3}
\pm 3\lambda_{2k,1}\lambda_{2k+1,2}
\pm 3\lambda_{2k+1,2}\lambda_{2k,1}.
$$
As a consequence of these computations,
we see that each of the anti-diagonal blocks of $[\cup_{\bf}]_{\mathscr{B}}$
are constant matrices:
$$
\pm M_{41}= \pm M_{14}=
\begin{bmatrix}
\begin{bmatrix}
& -1/2 \\
1/2 &
\end{bmatrix}
&
&
&\\
&
\begin{bmatrix}
& -1/2 \\
1/2 &
\end{bmatrix}
& & \\
& & \dots & \\
& & &
\begin{bmatrix}
& -1/2 \\
1/2 &
\end{bmatrix}
\end{bmatrix}
$$
and
$$
\pm M_{32}= \pm M_{23}=
\begin{bmatrix}
\begin{bmatrix}
& -3/2 \\
3/2 &
\end{bmatrix}
&
&
&\\
&
\begin{bmatrix}
& -3/2 \\
3/2 &
\end{bmatrix}
& & \\
& & \dots & \\
& & &
\begin{bmatrix}
& -1/2 \\
1/2 &
\end{bmatrix}
\end{bmatrix}.
$$

So $[\cup_{\bF}]_{\mathscr{B}}$
is an invertible matrix.
\end{proof}

The long root parabolic case is much simpler.
\section{Sagemath code}
\label{sec:sage}

\subsubsection{Proposition}
Let $V\subset B$ be the unipotent
radical of the Borel of $G_2$.
Let $g\in V(\bZp)$.
If $p > 5$, then $g^p = \Id$ mod $\varpi$.

\begin{proof}
Let $P \supset B$ be the short root parabolic.
Let $P = L \ltimes U$ be the Levi decomposition.
Let $\pi:P\to L$ be the quotient.
Say $\pi(g) = 
\begin{bmatrix}
1 & l \\
0 & 1
\end{bmatrix}
$.
Fix a projection $P\to U$.
Also fix a projection $U \to Z(U)$.
Say the projection of $g$ onto $U/Z(U) \cong \mathbb{A}^4$
via $P\to U\to U/Z(U)$ is 
$(u_0,u_1,u_2,u_3)$.
Say the projection of $g$ onto $Z(U)\cong \mathbb{A}^1$
via $P\to U\to Z(U)$
is $u_4$.

For simplicity, we write $g=(l;u_0,u_1,u_2,u_3;u_4)$.
We have, for any integer $q$,
\begin{eqnarray}
g^q&=&(q l; q u_0, -\frac{1}{2}q(q-1)u_0 l + q u_1, \nonumber\\
&&-\frac{1}{6}q(q-1)(2q-1)u_0 l^2 + q(q-1)u_1 l + q u_2,\nonumber\\
&&-\frac{1}{4}q^2(q-1)^2u_0l^3+\frac{1}{2}q(q-1)(2q-1)u_1l^2 
+\frac{3}{2}q(q-1)u_2 l + q u_3, q u_4;\nonumber\\
&&\frac{1}{120}(q-1)q(q+1)(3q^2-2)u_0^2l^3-\frac{1}{2}(q-1)q(q+1)(u_1^2+u_0u_2)l)\nonumber
\end{eqnarray}
It can be computed by hand, and can be verified by computer algebra system.
The Proposition follows from the above computation immediately.
\end{proof}

The SageMath source code for computing
is on the website \url{sharkoko.space}.


If we compute \hspace{2mm}
    \verb|cup_product_mod_p(5,4,4)|
\hspace{2mm}
in SageMath notebook,
we'll get an anti-lower-triangular matrix
in the sense of Lemma \ref{lem:cup-C-nondeg}.

\end{appendix}

\vspace{3mm}

\addcontentsline{toc}{section}{References}
\printbibliography

\end{document}